\documentclass[11pt]{amsart}
\usepackage[margin=1.35in]{geometry} 
\usepackage{amsmath,amsthm,amssymb,amsfonts}
\usepackage[utf8]{inputenc}
\usepackage[cyr]{aeguill}
\usepackage{graphicx}
\usepackage{nicefrac}
\usepackage{stmaryrd}
\usepackage{amsthm}
\usepackage{mathrsfs} 
\usepackage{faktor}

\usepackage{dsfont}
\usepackage{tikz}
\usepackage{xcolor}
\usepackage{yfonts}
\usepackage{colonequals}
\usepackage{mathabx}
\usepackage{tikz-cd}
\usepackage{hyperref}
\definecolor{darkgoldenrod}{rgb}{0.12, 0.28, 0.59}
\hypersetup{colorlinks=true, citecolor=darkgoldenrod, linkcolor=darkgoldenrod}
\usepackage[foot]{amsaddr}
\usepackage[french, english]{babel}

\DeclareFontFamily{U}{wncy}{}
\DeclareFontShape{U}{wncy}{m}{n}{<->wncyr10}{}
\DeclareSymbolFont{mcy}{U}{wncy}{m}{n}
\DeclareMathSymbol{\Sh}{\mathord}{mcy}{"58}

\newcommand{\gq}{{\mathfrak{q}}}

\newcommand{\gp}{{\mathfrak{p}}}

\newcommand{\im}{\mathrm{Im}}
\newcommand{\Fil}{\mathrm{Fil}}

\newcommand{\Gal}{{\mathrm{Gal}}}

\newcommand{\Sel}{{\mathrm{Sel}}}

\newcommand{\cris}{{\mathrm{cris}}}

\newcommand{\Hom}{\mathrm{Hom}}

\newcommand{\Q}{{\mathbf Q}}
\newcommand{\Z}{{\mathbf Z}}

\newcommand{\Dcris}{\mathbb{D}_{\mathrm{cris}}}

\newcommand{\col}{\mathrm{Col}}

\newcommand{\Qp}{\mathbf{Q}_p}
\newcommand{\Zp}{\mathbf{Z}_p}

\definecolor{Green}{rgb}{0.0, 0.5, 0.0}

\theoremstyle{definition}

\newtheorem{lemma}{Lemma}[section]
\newtheorem{definition}[lemma]{Definition}
\newtheorem{remark}[lemma]{Remark}
\newtheorem{corollary}[lemma]{Corollary}
\newtheorem{proposition}[lemma]{Proposition}
\newtheorem{conjecture}[lemma]{Conjecture}
\newtheorem{theorem}[lemma]{Theorem}
\newtheorem{lthm}{Theorem}

\newtheorem{notation}[lemma]{Notation}

\title[Functional equations over $\mathbf{Z}_p^2$-extensions]{Functional equations for supersingular abelian varieties over $\mathbf{Z}_p^2$-extensions}
\author[Cédric Dion]{Cédric Dion}
\address[C.~Dion]{D\'epartement de Math\'ematiques et de Statistique, Universit\'e Laval, Pavillion Alexandre-Vachon, 1045 Avenue de la M\'edecine, Qu\'ebec, QC, Canada G1V 0A6
}
\email{cedric.dion.1@ulaval.ca}
\date{}

\begin{document}


\maketitle

\begin{abstract}
Let $K$ be an imaginary quadratic field and $K_\infty$ be the $\mathbf{Z}_p^2$-extension of $K$. Answering a question of Ahmed and Lim, we show that the Pontryagin dual of the Selmer group over $K_\infty$ associated to a supersingular polarized abelian variety admits an algebraic functional equation. The proof uses the theory of $\Gamma$-system developed by Lai, Longhi, Tan and Trihan. We also show the algebraic functional equation holds for Sprung's chromatic Selmer groups of supersingular elliptic curves along $K_\infty$.

\medskip

\noindent\textsc{R\'esum\'e.}
Soit $K$ un corps quadratique imaginaire et $K_\infty$ l'unique $\Zp^2$-extension de $K$. Répondant à une question d'Ahmed et Lim, nous montrons que le dual de Pontryagin du groupe de Selmer sur $K_\infty$ associé à une variété abélienne supersingulière et polarisée admet une équation fonctionnelle algébrique. La preuve utilise la théorie des $\Gamma$-systèmes développée par Lai, Longhi, Tan et Trihan. Nous démontrons aussi qu'une équation fonctionnelle algébrique tient pour les groupes de Selmer chromatiques de Sprung sur $K_\infty$ de courbes elliptiques supersingulières. 
\end{abstract}

\section{Introduction}

Fix an odd prime number $p$. Let $V$ be a finite dimensional $\Qp$-vector space with a continuous $\Gal (\overline{\Qp}/\Qp)$-action. Suppose that $V$ is ordinary at $p$ in the sense of \cite{Gre89} and choose $T$ a $\Gal (\overline{\Qp}/\Qp)$-stable $\Zp$-lattice inside $V$. Let $A=V/T$. Let $\Q_\infty$ denote the cyclotomic $\Zp$-extension of $\Q$ with Galois group $\Gamma$. Write $S_A(\Q_\infty)$ for the Greenberg Selmer group of $A$ over $\Q_\infty$ as defined in \cite[Page 98]{Gre89}. Then, $S_A(\Q_\infty)$ is a module over the Iwasawa algebra $\Lambda = \Zp \llbracket \Gamma \rrbracket$. Let $\iota:\Lambda \to \Lambda$ be the involution induced by sending $\sigma \in \Gamma$ to $\sigma^{-1}$. For a $\Lambda$-module $M$ we define the $\Lambda$-module $M^\iota$ as the same set as $M$ but with action given by $\lambda \cdot m = \iota(\lambda)\cdot m$ for all $m\in M$ and $\lambda \in \Lambda$. Under the assumption that $S_A(\Q_\infty)$ is $\Lambda$-cotorsion (the Pontryagin dual $S_A(\Q_\infty)^\vee$ of $S_A(\Q_\infty)$ is a torsion $\Lambda$-module) and that the values of the associated $L$-functions $L_V(1)$ and $L_{V^\ast}(1)$ are critical values, Greenberg proved \cite[Theorem 2]{Gre89} that $S_A(\Q_\infty)$ and $S_{A^\ast}(\Q_\infty)^\iota$ have the same characteristic ideal where $A^\ast = \Hom (V,\Qp(1))/ \Hom (T,\Zp (1))$. He further makes the following conjecture \cite[Section 8 equation (66)]{Gre89}:

\begin{conjecture}\label{Greenberg conjecture}
One should expect that $S_A(\Q_\infty)^\vee \sim S_{A^\ast}(\Q_\infty)^{\vee,\iota}$ when $V$ is $p$-critical.
\end{conjecture}

The symbol $\sim$ denotes pseudo-isomorphisms, i.e., $\Lambda$-homomorphisms with finite kernels and cokernels. When $A=E[p^\infty]$ where $E$ is an elliptic curve over $\Q$ with good ordinary reduction or multiplicative reduction at $p$, conjecture \ref{Greenberg conjecture} can be deduced from \cite[Theorem 2]{Gre89}. Such a result should be seen as an algebraic analogue of the familiar functional equation $L_p(E,T) = (*) \cdot L_p(E,\frac{1}{1+T}-1)$ for the $p$-adic $L$-function associated to $E$ where $(*)$ is an explicit factor.

Over the years, algebraic functional equations have been proven for Selmer groups attached to many kinds of Galois representations over various extensions. Let $K$ be a number field unramified at $p$ and let $T$ be a $\Zp$-lattice with a continuous $\Gal (\overline{\Q}/\Q)$-action and let $K_\infty$ be a  $\Zp^d$-extension of $K$. In the case where $T$ is the $p$-adic Tate module of an elliptic curve over $\Q$ with supersingular reduction at primes above $p$ with $a_p=0$ and $K_\infty=\Q_\infty$ is the cyclotomic $\Zp$-extension of $\Q$, we can consider Kobayashi's $\pm$-Selmer groups $\Sel_\pm(E/\Q_\infty)$. In this setting, Kim \cite{Kim08} proved an analogue of conjecture \ref{Greenberg conjecture}, namely that $\Sel_\pm(E/\Q_\infty)^\vee \sim \Sel_\pm(E/\Q_\infty)^{\vee,\iota}$ by adapting Greenberg's technique. When $K_\infty$ is a $\Zp^d$-extension with $d\geq 2$, Kim constructed multi-signed Selmer groups $\Sel_{\vec{s}}(E/K_\infty)$ generalizing Kobayashi's plus and minus Selmer groups. In \cite{AL21}, Ahmed and Lim proved that $\Sel_{\vec{s}}(E/K_\infty)^\vee \sim \Sel_{\vec{s}}(E/K_\infty)^{\vee,\iota}$. Their result is actually more general, they allow mixed reduction type at the primes above $p$.

For example of results when $T$ does not come from an elliptic curve, one can look at the work of Lei and Ponsinet \cite{LP17}. Suppose that $T$ is crystalline at primes above $p$ and satisfies some technical conditions. Büyükboduk and Lei \cite{BL17} constructed multi-signed Selmer groups $\Sel_I(T/K_\infty)$ for $T$ over the cyclotomic $\Zp$-extension of $K$ by making use of $p$-adic Hodge theory. The functional algebraic equation for $\Sel_I(T/K_\infty)$ was then proven in \cite{LP17}.

Let now $A$ be an abelian variety defined over a number field $K$ with potentially ordinary reduction at every places of a finite ramification locus $S$. Let $K_\infty/K$ be a $\Zp^d$-extension and let $\Sel_{p^\infty}(A/K_\infty)$ be the $p^\infty$-Selmer group of $A$ over $K_\infty$ defined by means of flat cohomology \cite[Definition 4.1.1]{LLTT}. Suppose that $\Sel_{p^\infty}(A/K_\infty)^\vee$ is torsion over the Iwasawa algebra $\Zp \llbracket \Gal(K_\infty/K)\rrbracket$. Then, by \cite[Proposition 4.3.4]{LLTT}, $\Sel_{p^\infty}(A/K_\infty)^\vee \sim \Sel_{p^\infty}(A^t/K_\infty)^{\vee,\iota}$ where $A^t$ is the dual abelian variety. This result is a byproduct of a much more general theory developed in \cite{LLTT} that can be used to prove functional equation for Iwasawa modules as long as they are part of a \textit{$\Gamma$-system}.
In the introduction of \cite{AL21}, Ahmed and Lim ask whether or not the machinery of $\Gamma$-systems can be used to study non-ordinary motives with Hodge-Tate weights $0$ and $1$ over $\Zp^d$-extensions of number fields in a manner similar to what they did for supersingular elliptic curves. In this paper, we give a partial answer to this question. Partial in the sense that we restrict ourselves to $\Zp^2$-extension of quadratic imaginary fields, we do not treat the case of mixed reduction type and we only consider motives arising from abelian varieties satisfying additional conditions. 

To state our result, we need to introduce more notation. Let $K$ be an imaginary quadratic field where $(p)=\gp \gp^c$ splits. Let $K_\infty$ be the unique $\Zp^2$-extension of $K$ with Galois group $\Gamma \cong \Zp^2$. Let $A$ be a polarized abelian variety defined over $K$ of dimension $g$ with supersingular reduction at both primes over $p$. Let $\mathbb{D}_{\cris,\gp}(T)$ (resp. $\mathbb{D}_{\cris,\gp^c}(T)$) be the Dieudonné module of the $p$-adic Tate module of $A$ viewed as a representation of $G_{K_\gp}$ (resp. $G_{K_{\gp^c}}$). Let $\{v_{\gp,1},\ldots, v_{\gp,2g}\}$ be a basis of $\mathbb{D}_{\cris,\gp}(T)$ such that $\{v_{\gp,1},\ldots, v_{\gp,g}\}$ generates $\Fil^0\mathbb{D}_{\cris,\gp}(T)$. A basis with this property is called Hodge-compatible. The matrix of the Frobenius operator $C_{\varphi,\gp}$ acting on $\mathbb{D}_{\cris,\gp}(T)$ with respect to the chosen basis takes the form 
$$
C_{\varphi,\gp} = C_\gp \left[
\begin{array}{c|c}
I_{g} & 0 \\
\hline
0 & \frac{1}{p} I_{g}
\end{array}
\right]
$$
for some $C_\gp \in \mathrm{GL}_{2g}(\Zp)$ since the Hodge-Tate weights of $T$ are $0$ and $1$. Let $I$ be the subset $\{1,\ldots,2g\} \subseteq \{1,\ldots,4g\}$ and $I^c$ its complement. We can construct Selmer groups $\Sel_{I}(A[p^\infty]/K_\infty)$ and $\Sel_{I^c}(A^t[p^\infty]/K_\infty)$ (see section \ref{definition}) depending on $I$ and the choice of Hodge-compatible bases for both $\mathbb{D}_{\cris,\gp}(T)$ and $\mathbb{D}_{\cris,\gp^c}(T)$. We remark that both bases may be chosen independently, as no relation between them is required in the construction of the Selmer groups. 

\begin{lthm}
Suppose that $\Sel_I(A[p^\infty]/K_\infty)^\vee$ is torsion over $\Zp \llbracket \Gamma \rrbracket$. Suppose that the matrices $C_{\varphi,\gp}$ and $C_{\varphi,\gp^c}$ are block anti-diagonal and that $p$ does not divide the class number of $K$. Then,
$$
\Sel_I(A[p^\infty]/K_\infty)^\vee \sim \Sel_{I^c}(A^t[p^\infty]/K_\infty)^{\vee,\iota}.
$$
\end{lthm}

Let us briefly explain the steps of the proof. Since we assume $\Sel_I(A[p^\infty]/K_\infty)^\vee$ to be torsion, there is a pseudo-isomorphism
$$
\Sel_I(A[p^\infty]/K_\infty)^\vee \to \bigoplus_{i=1}^m \Zp \llbracket \Gamma \rrbracket /(\xi_i^{r_i})
$$
where each $\xi_i$ is irreducible and $r_i$ are non-negative integers. Let
$$
[\Sel_I(A[p^\infty]/K_\infty)^\vee]\colonequals \bigoplus_{i=1}^m \Zp \llbracket \Gamma \rrbracket /(\xi_i^{r_i})
$$
and define $[\Sel_I(A[p^\infty]/K_\infty)^\vee]_\text{si}$ as the sum over the $\xi_i$ which are simple element (to be defined in section \ref{gamma-systems}) and $[\Sel_I(A[p^\infty]/K_\infty)^\vee]_\text{ns}$ as its complement. The proof of theorem A goes in two steps. We first show that $[\Sel_I(A[p^\infty]/K_\infty)^\vee]_\text{ns} = [\Sel_{I^c}(A^t[p^\infty]/K_\infty)^\vee]_\text{ns}^\iota$ (theorem \ref{simple parts}) by using the machinery of $\Gamma$-system as in \cite{LLTT} and \cite{AL21}. For the second step, namely showing that $[\Sel_I(A[p^\infty]/K_\infty)^\vee]_\text{si} = [\Sel_{I^c}(A^t[p^\infty]/K_\infty)^\vee]_\text{si}^\iota$ (see corollary \ref{simple parts}), we can not use the arguments of \cite{AL21} since unlike elliptic curves, abelian varieties are not necessarily self-dual. To circumvent this problem, the main difficulty will be to construct a $\Zp \llbracket \Gamma \rrbracket$-homomorphism $$\Sel_I(A[p^\infty]/K_\infty)\to\Sel_{I^c}(A^t[p^\infty]/K_\infty)$$ in order to put ourselves in a situation where we can apply \cite[Corollary 4.3.2]{LLTT}. To do that, we need to analyze the behaviour of the basis of $\mathbb{D}_{\cris,\gp}(T)$ when we hit it with an isogeny $\alpha:A \to A^t$. This is done in section \ref{sct:Abelian}.

Finally, we show the functional equation in the case when $A=E$ is an elliptic curve over $K$ with supersingular reduction at both primes above $p$ and the Selmer groups considered are Sprung's $\sharp/\flat$-$\sharp/\flat$ Selmer groups. Since for elliptic curves $E^t=E$, we are able to mostly follow the argument in \cite{AL21}. However, one still need to show that the local conditions defining those Selmer groups are their own orthogonal complement with respect to the Tate pairing in order to use the theory of $\Gamma$-systems. This is done via a comparison theorem relating multi-signed Coleman maps to $\sharp/\flat$-Coleman maps. We obtain:

\begin{lthm}
Let $\star,\circ \in \{\sharp,\flat\}$ and let $\Sel^{\star\circ}(E/K_\infty)$ be the chromatic Selmer group defined in section \ref{chromatic}. Then,
$$
\Sel^{\star\circ}(E/K_\infty)^\vee \sim \Sel^{\star\circ}(E/K_\infty)^{\vee,\iota}.
$$
\end{lthm}

\noindent \textbf{Acknowledgements} The author thanks his Ph.D. advisor Antonio Lei for suggesting the problem and for his continuous support while investigating this topic. The author is also grateful to Jeffrey Hatley and Eyal Goren for answering our questions. The author also thanks Meng Fai Lim for his suggestions. The author also extends his thanks to Anthony Doyon for many helpful discussions. Finally, many thanks to Luochen Zhao for pointing out an inaccuracy in the original proof of lemma \ref{Coleman inverse limit} and to the referee for many constructive comments that helped improve the quality of the paper. The author's research is supported by the Canada Graduate Scholarships – Doctoral program from the Natural Sciences and Engineering Research Council of Canada.

\section{Preliminaries}

\subsection{Global and local setup}

Let $p \geq 3$ be a prime number and $K$ be an imaginary quadratic field where $(p)=\mathfrak{p}\mathfrak{p}^c$ splits with ring of integers $\mathcal{O}_K$. Here, $c$ is the complex conjugation. We will always use the symbol $\mathfrak{q}$ to mean an element of $\{\mathfrak{p},\mathfrak{p}^c \}$. Let $K_\infty$ be the unique $\mathbf{Z}_p^2$-extension of $K$ and $\Gamma:= \text{Gal}(K_\infty /K) \cong \mathbf{Z}_p^2$. Let $\Gamma_n \colonequals \Gamma^{p^n}$ and $K_n \colonequals K_\infty^{\Gamma_n}$. We write $K^\text{cyc}/K$ and $K^\text{ac}/K$ for the cyclotomic and anticyclotomic $\mathbf{Z}_p$-extensions contained in $K_\infty$ respectively. Note that $K_\infty$ is the compositum $K^\text{cyc}K^\text{ac}$. Write $\Gamma^\text{cyc} := \text{Gal}(K^\text{cyc}/K)$ and $\Gamma^\text{ac}:= \text{Gal}(K^\text{ac}/K)$. Let $\mu_{p^n}$ denote the set of $p^n$th roots of unity, $\mu_{p^\infty} \colonequals \bigcup_{n \geq 1} \mu_{p^n}$ and $\zeta_{p^n}$ be a primitive $p^n$th root of unity. We also let $\Gamma_0^\text{cyc}:= \text{Gal}(K(\mu_{p^\infty})/K) \cong \Gamma^\text{cyc} \times \Delta$ where $\Delta$ is a cyclic group of order $p-1$. Let $\chi$ denote the $p$-adic cyclotomic character. Let $F_\infty$ be the unramified $\mathbf{Z}_p$-extension of $\mathbf{Q}_p$, $U=\text{Gal}(F_\infty/\mathbf{Q}_p)$ and let $\widehat{F_\infty}$ denote the completion of $F_\infty$. Let $k_\infty$ be the compositum of $F_\infty$ with the cyclotomic $\mathbf{Z}_p$-extension of $\mathbf{Q}_p$ with Galois group denoted by $\Gamma_p \colonequals \Gal (k_\infty /\Qp)$. Write $\Gamma_{p,n}\colonequals (\Gamma_p)^{p^n}$ and define $k_n$ to be $k_\infty^{\Gamma_{p,n}}$. If $F$ is a finite unramified extension of $\Qp$, we write $F^\mathrm{cyc}$ for its cyclotomic $\Zp$-extension. We shall also write $\Gamma^\mathrm{cyc}_F$ (resp. $\Gamma^\mathrm{cyc}_{F,0}$) for $\Gal(F^\mathrm{cyc}/F)$ (resp. $\Gal(F(\mu_{p^\infty})/F)$). If $L$ is a local field or a number field, write $G_L$ for its absolute Galois group $\text{Gal}(\overline{L}/L)$ where $\overline{L}$ is an algebraic closure of $L$. If $\theta$ is a finite order character of $\Gamma_p$ with cyclotomic conductor $p^{n+1}$, we let $\tau(\theta) \colonequals \sum_{\sigma \in \Gal(\Qp(\mu_{p^{n+1}})/\Qp)}\theta(\sigma)\zeta_{p^{n+1}}^\sigma$ denote its Gauss sum. Let $\mathbf{C}_p$ denote the completion of an algebraic closure of $\Qp$. \bigbreak

For a profinite group $P$, we denote by $\Lambda(P)$ the Iwasawa algebra $\varprojlim_{B} \mathbf{Z}_p[P/B]$ where $B$ runs over the normal open subgroups of $P$. Let $\mathcal{H}(\Gamma_0^\text{cyc})$ be the set of power series
$$
\sum_{n\geq 0, \sigma \in \Delta} c_{n,\sigma}\cdot \sigma \cdot (\gamma_0 -1 )^n
$$
with coefficients in $\mathbf{Q}_p$ such that $\sum_{n\geq 0}c_{n,\sigma}X^n$ converges on the open unit disk for all $\sigma \in \Delta$. Here, $\gamma_0$ is a topological generator of $\Gamma^{\mathrm{cyc}}$. Let $\mathcal{H}(\Gamma^\text{cyc})$ be the set of power series
$$
\sum_{n\geq 0} c_{n}\cdot (\gamma_0 -1 )^n
$$
with coefficients in $\mathbf{Q}_p$ such that $\sum_{n\geq 0}c_{n}X^n$ converges on the open unit disk. Let $\gamma_1$ be a topological generator of $U$. Similarly, $\mathcal{H}_{\widehat{F}_\infty}(\Gamma_p)$ is the set of power series in $\gamma_0-1$ and $\gamma_1-1$ with coefficients in $\widehat{F}_\infty$ converging on the closed balls with radius smaller or equal to $r$ for all $r < 1$. We shall identify $\Lambda (\Gamma^\text{cyc})$ with $\Zp \llbracket X \rrbracket$ by sending $\gamma_0-1$ to the variable $X$.\bigbreak

Let $\mathcal{M}_{/K}$ be a motive defined over $K$ in the sense of \cite{FPR94} and $\mathcal{M}_p$ its $p$-adic realization. Let $T$ be a $G_K$-stable $\mathbf{Z}_p$-lattice inside $\mathcal{M}_p$. If $L$ is a local field and $L^\prime /L$ is a $p$-adic Lie extension, we write $H^1_\text{Iw}(L^\prime,\bullet)$ for the first Iwasawa cohomology group $\varprojlim_{L^{\prime\prime}} H^1(L^{\prime\prime},\bullet)$ where $L^{\prime\prime}$ runs through all finite subextension of $L^\prime /L$. We shall denote by $T^\dagger \colonequals \text{Hom}(T,\mu_{p^\infty})$ the Cartier dual of $T$ and by $T^\ast(1) \colonequals \text{Hom}(T,\mathbf{Z}_p(1))$ the Tate dual of $T$. Let $M$ be any $\Lambda(\Gamma_0^\mathrm{cyc})$-module and $\eta$ be a Dirichlet character modulo $p-1$. Write $e_\eta \colonequals \frac{1}{p-1}\sum_{\sigma \in \Delta} \eta(\sigma)^{-1}\cdot \sigma \in \Zp[\Delta]$ for the idempotent corresponding to $\eta$. The $\eta$-isotypic component of $M$ is defined to be $e_\eta \cdot M$ and is denoted by $M^\eta$. Then $M^\eta$ admits the structure of a $\Lambda(\Gamma^\mathrm{cyc})$-module. If $\eta$ is the trivial character modulo $p-1$, we denote $M^\eta$ by $M^\Delta$.

\subsection{Dieudonné modules}

Let $\mathbf{A}_{\mathbf{Q}_p}^+ \colonequals \mathbf{Z}_p \llbracket \pi \rrbracket$ where $\pi$ is a formal variable. The ring $\mathbf{A}_{\mathbf{Q}_p}^+$ is equipped with a Frobenius action $\varphi: \pi \to (1+\pi)^p-1$ and a $\Gamma_0^\text{cyc}$-action $\sigma: \pi \to (1+\pi)^{\chi(\sigma)}-1$. Let $\mathbf{A}_{\mathbf{Q}_p}$ be the $p$-adic completion of $\mathbf{Z}_p \llbracket \pi \rrbracket [\pi^{-1}]$. Let
$$
\widetilde{\mathbf{E}} \colonequals \varprojlim_{x \mapsto x^p} \mathbf{C}_p
$$
which is a field of characteristic $p$. Define $\widetilde{\mathbf{A}}$ to be the ring of Witt vectors with coefficient in $\widetilde{\mathbf{E}}$ and $\widetilde{\mathbf{B}}$ to be $\widetilde{\mathbf{A}}[1/p]$. Let $\mathbf{B}$ be the completion with respect to the $p$-adic topology of the maximal unramified extension of the field $\mathbf{A}_{\Qp}$ inside $\widetilde{\mathbf{B}}$. Finally, we let $\mathbf{A} \colonequals \mathbf{B} \cap \widetilde{\mathbf{A}}$. For more details on period rings, see \cite[Section I.2]{Be04}.

Suppose that \bigbreak
\noindent \textbf{(H.crys)} $\mathcal{M}_p$ is crystalline at $\mathfrak{p}$ and $\mathfrak{p}^c$. \bigbreak

\noindent Let $d:= \text{dim}_{\mathbf{Q}_p}(\text{Ind}_{K/\mathbf{Q}}\mathcal{M}_p)$ and let $d_\pm := \text{dim}_{\mathbf{Q}_p}(\text{Ind}_{K/\mathbf{Q}}\mathcal{M}_p)^{c=\pm 1}$. Let $\mathbb{D}_{\cris}(\mathcal{M}_p)$ be $(\mathbb{B}_\text{cris}\otimes_{\mathbf{Q}_p} \mathcal{M}_p)^{G_{\mathbf{Q}_p}}$ where $\mathbb{B}_\text{cris}$ is the crystalline period ring defined by Fontaine. It admits the structure of a filtered $\varphi$-module. Write $\mathbb{D}(T)$ for the Dieudonné module $(\mathbf{A} \otimes_{\mathbf{Z}_p} T)^H$ where $H$ is the kernel of the cyclotomic character $G_{\mathbf{Q}_p} \to \mathbf{Z}_p^\times$. The module $\mathbb{D}(T)$ is a free $\mathbf{A}_{\mathbf{Q}_p}$-module of rank $d$ equipped with a Frobenius and an action of $\Gamma_0^\text{cyc}$. Let $\mathbb{N}(T)$ be the Wach module of $T$ whose existence and properties are shown in \cite[prop. 2.1.1]{Be04}. It is a free $\mathbf{A}_{\mathbf{Q}_p}^+$-module of rank $d$ and a $\mathbf{A}_{\mathbf{Q}_p}^+$-submodule of $\mathbb{D}(T)$. The Wach module $\mathbb{N}(T)$ is stable under the action of $\Gamma_0^\text{cyc}$ and is also stable under $\varphi$ provided that the Hodge-Tate weights of $\mathcal{M}_p$ are smaller or equal to $0$. Write $\psi$ for a left inverse of $\varphi$. If $\mathcal{M}_p$ has nonnegative Hodge-Tate weights and no quotient isomorphic to $\Qp$, then $\mathbb{N}(T)^{\psi=1}=\mathbb{D}(T)^{\psi=1}$. Furthermore, the quotient $\mathbb{N}(T)/\pi \mathbb{N}(T)$ is identified with a $\mathbf{Z}_p$-lattice of $\mathbb{D}_\text{cris}(\mathcal{M}_p)$. We denote by $\mathbb{D}_\text{cris}(T)$ this $\mathbf{Z}_p$-lattice. It is equipped with a filtration of $\mathbf{Z}_p$-modules $\{ \Fil^i \mathbb{D}_\text{cris}(T) \}_{i \in \mathbf{Z}}$ and a Frobenius operator $\varphi$. If we suppose that \bigbreak
\noindent \textbf{(H.HT)} the Hodge-Tate weights of $\mathcal{M}_p$ are either $0$ or $1$,\bigbreak
\noindent the filtration takes the form
$$
\Fil^i \mathbb{D}_\text{cris}(T) = \begin{cases} 0 & \text{if $i \geq 1$,} \\ \mathbb{D}_\text{cris}(T) & \text{if $i \leq -1$.} \end{cases}
$$
Note that $\mathbb{D}_\text{cris}(\mathcal{M}_p) = \mathbb{D}_\text{cris}(T) \otimes_{\mathbf{Z}_p} \mathbf{Q}_p$. We also make the following assumptions: \bigbreak
\noindent \textbf{(H.Frob)} The slopes of the Frobenius on the Dieudonné module $\mathbb{D}_\text{cris}(\mathcal{M}_p)$ lie inside $(0,-1]$ and that $1$ is not an eigenvalue; \\
\textbf{(H.P)} $d_+ = d_-$ and $\text{dim}_{\mathbf{Q}_p} \text{Fil}^0 \mathbb{D}_\text{cris}(\mathcal{M}_p)=d_-$. \bigbreak
From now on, we shall always assume that \textbf{(H.crys)}, \textbf{(H.HT)}, \textbf{(H.Frob)} and \textbf{(H.P)} are valid. In particular, \textbf{(H.P.)} implies that $d$ is even and that $\Fil^0 \mathbb{D}_\text{cris}(T)$ is of rank $d/2$ over $\Zp$. Choose a $\mathbf{Z}_p$-basis $\{v_1,\ldots,v_d \}$ of $\mathbb{D}_\text{cris}(T)$ such that $\{ v_1,\ldots, v_{d_-} \}$ is a $\mathbf{Z}_p$-basis of the submodule $\Fil^0\mathbb{D}_\text{cris}(T)$. Such a basis is called Hodge-compatible. The matrix of $\varphi$ with respect to this basis is of the form
$$
C_\varphi = C\left[
\begin{array}{c|c}
I_{d_-} & 0 \\
\hline
0 & \frac{1}{p} I_{d_+}
\end{array}
\right]
$$
for some $C \in \text{GL}_d(\mathbf{Z}_p)$ and where $I_n$ is the identity $n \times n$ matrix. There is a natural pairing
$$
[\cdot,\cdot]:\mathbb{D}_\text{cris}(T) \times \mathbb{D}_\text{cris}(T^\ast(1)) \to \mathbb{D}_\text{cris}(\mathbf{Z}_p(1)) \cong \mathbf{Z}_p
$$
with respect to which $\text{Fil}^i\mathbb{D}_\text{cris}(T^\ast(1))$ is the orthogonal complement of $\text{Fil}^{-i}\mathbb{D}_\text{cris}(T)$ and $\varphi^{-1}$ is the adjoint of $p\varphi$. To see this, remark that by \cite[Example 9.1.12]{BC}, we have $\Dcris (\Zp (1)) = \Zp \cdot t^{-1}$ where $t=\log([\epsilon])$ is an element of $\mathbb{B}_{\mathrm{cris}}$ satisfying the relations $\varphi(t)=pt$ and $\sigma \cdot t = \chi(\sigma)t$ for all $\sigma \in G_{\Qp}$. Thus, 
\begin{align}\label{dual operator}
[\varphi(a\otimes x),b\otimes f] &= \varphi(a)b \cdot f(x) = \varphi( a\varphi^{-1}(b) \cdot f(x)) = ap^{-1}\varphi^{-1}(b)f(x)\\
&= [a\otimes x, p^{-1}\varphi^{-1}(b \otimes f)]. \nonumber
\end{align}

If $F$ is a finite unramified extension of $\Qp$, one can define both the Wach module $\mathbb{N}_F(T)$ and crystalline module $\Dcris(F,T)$ of $T$ over $F$ satisfying $\mathbb{N}_F(T) = \mathbb{N}(T) \otimes_{\Zp}\mathcal{O}_F$ and $\Dcris(F,T) = \Dcris(T)\otimes_{\Zp}\mathcal{O}_F$.

\section{Multi-signed Coleman maps}\label{Coleman}

In this section, we review the construction of one-variable multi-signed Selmer groups defined by Büyükboduk and Lei. For a more detailed description, see \cite{BL17}. After that, we define two-variable multi-signed Selmer groups in the spirit of \cite{BL21} as was done in \cite{DR21}. Lastly, we review orthogonality properties satisfied by local conditions used to define multi-signed Selmer groups following Ponsinet \cite{Pon20} and Lei--Ponsinet \cite{LP17}.

\subsection{Definition}\label{definition}

Let's recall the definition of one-variable Coleman maps from \cite{BL17} as it will be needed for the construction of the two-variable counterpart. Let $F/\Qp$ be a finite unramified extension of degree $r$, $F^\text{cyc}$ be the cyclotomic $\Zp$-extension of $F$ and $F^n$ be the unique subextension of $F^\text{cyc}$ with $[F^n:F]=p^n$. Let $\mathcal{M}_p$ be a $\Qp$-vector space of dimension $d$ with a continuous action of $G_{\Qp}$ satisfying the hypotheses of the previous section. Fix $T$ a $G_{\Qp}$-stable $\Zp$-lattice inside $\mathcal{M}_p$. For the rest of the paper, we fix $\{v_1,\ldots, v_{d}\}$ a Hodge-compatible $\Zp$-basis of $\Dcris (T)$ such that $\{v_1,\ldots,v_{d_-}\}$ is a basis of $\Fil^0 \Dcris (T)$. Write $\Dcris(F,T)=\Dcris (T)\otimes_{\Zp} \mathcal{O}_F$ for the Dieudonné module of $T$ seen as a representation of $G_F$. Then, $\{v_1,\ldots, v_{d}\}$ is a Hodge-compatible $\mathcal{O}_F$-basis of $\Dcris (F,T)$ and the matrix of $\varphi$ on $\Dcris (F,T)$ with respect to this basis is the same as the one described in the previous section. For $n \geq 1$, let $\Phi_{p^n}(1+X)$ be the cyclotomic polynomial $\sum_{i=0}^{p-1}(1+X)^{ip^{n-1}}$. Define the matrices
$$
C_n \colonequals \left[
\begin{array}{c|c}
I_{d_-} & 0 \\
\hline
0 & \Phi_{p^n}(1+X) I_{d_+}
\end{array}
\right]
C^{-1}
$$
and $M_n \colonequals (C_\varphi)^{n+1}C_n \cdots C_1$. In \cite[Proposition 2.5]{BL17}, it is shown that the sequence $\{ M_n \}_{n\geq 1}$ converges entry-wise with respect to the sup-norm topology on $\mathcal{H}(\Gamma^\text{cyc})$ to a $d \times d$ matrix with entries in $\mathcal{H}(\Gamma^\text{cyc})$ which we call $M_{T}$.

\begin{remark}
Let $\{w_1,\ldots,w_r\}$ be a $\Zp$-basis of $\mathcal{O}_F$. This gives rise to a $\Zp$-basis of $\Dcris (F,T)$, namely $\{v_i\otimes w_j : 0\leq i \leq d, 0\leq j \leq r \}$. The matrices $C_n$ defined in \cite{BL17} are $rd \times rd$ matrices since they work with $\Zp$-bases instead of $\mathcal{O}_F$-bases. 
\end{remark}

\begin{proposition}\label{determinant}
Up to a constant in $\mathbf{Z}_p^\times$, $\det M_T$ is equal to $\left( \frac{\log(1+X)}{pX} \right)^{d_+}$.
\end{proposition}

\begin{proof}
See \cite[Proposition 2.5]{BL17}.
\end{proof}

Let $h^1_{\mathrm{Iw},T}: \mathbb{N}_F(T)^{\psi=1} \xrightarrow{\sim} H^1_\mathrm{Iw}(F(\mu_{p^\infty}),T)$ be Fontaine's isomorphism \cite[Proposition I.8]{Be03}. Let 
$$
\mathcal{L}_{T,F}:H^1_\mathrm{Iw}(F(\mu_{p^\infty}),T) \to \mathcal{H}(\Gamma_{F,0}^\text{cyc})\otimes_{\Zp} \Dcris(F,T)
$$
be Perrin-Riou's $p$-adic regulator defined as the composition $(\mathfrak{M}^{-1}\otimes 1)\circ (1-\varphi)\circ (h^1_{\mathrm{Iw},T})^{-1}$ where $\mathfrak{M}$ is the Mellin transform \cite[Definition 3.4]{LLZ11}. Then, by \cite[Theorem 2.13]{BL17} there exists a unique $\mathcal{O}_F \otimes_{\Zp}\Lambda(\Gamma_{F,0}^\text{cyc})$-homomorphism
$$
\col_{T,F}: H^1_\mathrm{Iw}(F(\mu_{p^\infty}),T) \to \mathcal{O}_F \otimes_{\Zp}\Lambda(\Gamma_{F,0}^\text{cyc})^{\oplus d}
$$
such that for all $z \in H^1_\mathrm{Iw}(F(\mu_{p^\infty}),T)$ we have the decomposition \cite[Corollary 2.14]{BL17}
\begin{equation}\label{decomposition une variable}
\mathcal{L}_{T,F}(z) = (v_1 \cdots v_d)\cdot M_T \cdot \col_{T,F}(z).
\end{equation}

As discussed in \cite[Section 2.7]{LZ14}, $h^1_{\mathrm{Iw},T}$ commutes with the action of the larger group $\Gal (F(\mu_{p^\infty})/\Qp)$. One can also see that $1-\varphi$ and $\mathfrak{M}^{-1}\otimes 1$ also commute with this action making $\mathcal{L}_{T,F}$ into a homomorphism of $\mathcal{O}_F \otimes_{\Zp}\Lambda ( \Gal (F(\mu_{p^\infty})/\Qp))$-modules. Thus, $\col_{T,F}$ also commutes with the action of $\Gal (F(\mu_{p^\infty})/\Qp)$.

\begin{remark}
\cite[Theorem 2.13]{BL17} gives a decomposition with respect to a $\Zp$-basis of $\Dcris (F,T)$. However, one may follow the same proof to show that the decomposition \eqref{decomposition une variable} with respect to our chosen $\mathcal{O}_F$-basis holds. To do this, replace $\Lambda_n$ by $\mathcal{O}_F \otimes_{\Zp} \Lambda_n$ in the argument of \textit{loc. cit.} and remark that, since $C_\varphi$ is defined over $\Qp$, the matrix representation of the powers $\varphi^{-n}$ of the semilinear operator $\varphi$ is the usual matrix product $C_\varphi^{-n}$.
\end{remark}

\begin{remark}
The construction of the logarithmic matrix $M_T$ and the Coleman map $\col_{T,F}$ depend on the choice of a Hodge-compatible $\Zp$-basis $\{v_1,\ldots,v_d\}$ of $\Dcris (T)$.
\end{remark}

Suppose that \bigbreak
\noindent \textbf{(H.F)} the group $H^0(F(\mu_{p^\infty}), T^\ast(1)^\dagger)$ is finite. \bigbreak
In particular, the groups $H^0(F(\mu_{p^n}),T)$ are trivial for all $n\geq 1$ since $T$ is isomorphic to $\varprojlim_n T^\ast(1)^\dagger [p^n]$. By \textbf{(H.F)} and inflation-restriction, we have the family of isomorphisms
$$
H^1(F^n,T) \cong H^1(F(\mu_{p^n}),T)^\Delta.
$$
Thus,
$$
H^1_\mathrm{Iw}(F^\mathrm{cyc},T) \cong H^1_\mathrm{Iw}(F(\mu_{p^\infty}),T)^\Delta.
$$

\begin{notation}
By abuse of notation, we also write the restriction $\mathcal{L}_{T,F} \lvert_{H^1_\mathrm{Iw}(F(\mu_{p^\infty}),T)^\Delta}$ as $\mathcal{L}_{T,F}$. With this convention, we get a regulator map
$$
\mathcal{L}_{T,F}:H^1_\mathrm{Iw}(F^\mathrm{cyc},T) \to \mathcal{H}(\Gamma_F^\text{cyc})\otimes_{\Zp} \Dcris(F,T)
$$
and Coleman map
$$
\col_{T,F}: H^1_\mathrm{Iw}(F^\mathrm{cyc},T) \to \mathcal{O}_F \otimes_{\Zp}\Lambda(\Gamma_F^\text{cyc})^{\oplus d}.
$$
For the rest of the paper, we will only consider the isotypic component of the trivial character of $\Delta$ for all $\Lambda(\Gamma^\mathrm{cyc}_{F,0})$-modules involved.
\end{notation}


We now extend this construction to the two-variable setting following \cite{DR21}. Fix $T$ a $G_{\Qp}$-stable $\mathbf{Z}_p$-lattice inside the $p$-adic realization $\mathcal{M}_p$ of a motive $\mathcal{M}_{/K}$. Let $S_{F_\infty / \mathbf{Q}_p} \subseteq \mathcal{O}_{\widehat{F_\infty}}\llbracket U \rrbracket$ be the Yager module in the sense of Loeffler--Zerbes \cite[Section 3.2]{LZ14}. It is a free $\Lambda( U )$-module of rank $1$. Fix $\{\Omega_{\mathbf{Q}_p}\}$ a basis of $S_{F_\infty / \mathbf{Q}_p}$. Define $\mathbb{N}_{F_\infty}(T)$ as the completed tensor product of the Wach module of $T$ with the Yager module: $\mathbb{N}(T)\widehat{\otimes}_{\Zp} S_{F_\infty/\Qp}$. Then, there is an isomorphism \cite[Proposition 4.5]{LZ14}
$$
h^1_{\infty,T}:\mathbb{N}_{F_\infty}(T)^{\psi=1} \to H^1_\mathrm{Iw}(F_\infty(\mu_{p^\infty}),T).
$$
Let $\mathcal{L}_{T,k_\infty}$ be the two-variable big logarithm map of Loeffler--Zerbes (also called the two-variable $p$-adic regulator) \cite[Definition 4.6]{LZ14}
$$
\mathcal{L}_{T,k_\infty}: H^1_\mathrm{Iw}(F_\infty(\mu_{p^\infty}),T) \to \Omega_{\mathbf{Q}_p} \cdot \left( \mathcal{H}(\Gamma_0^\text{cyc})\widehat{\otimes}\Lambda(U) \right) \otimes_{\mathbf{Z}_p} \mathbb{D}_\text{cris}(T)
$$
defined as the composition of $$(\varphi^\ast \mathbb{N}(T))^{\psi=0}\widehat{\otimes}_{\Zp}S_{F_\infty/\Qp} \to \Omega_{\mathbf{Q}_p} \cdot \left( \mathcal{H}(\Gamma_0^\text{cyc})\widehat{\otimes}\Lambda(U) \right) \otimes_{\mathbf{Z}_p} \mathbb{D}_\text{cris}(T)$$ with $(1-\varphi)\circ (h^1_{\infty,T})^{-1}$. Again, we take the isotypic component of the trivial character and see $\mathcal{L}_{T,k_\infty}$ as a map
$$
H^1_\mathrm{Iw}(k_\infty,T)\to \Omega_{\mathbf{Q}_p} \cdot \left( \mathcal{H}(\Gamma^\text{cyc})\widehat{\otimes}\Lambda(U) \right) \otimes_{\mathbf{Z}_p} \mathbb{D}_\text{cris}(T).
$$
Consider the extension $F_m \colonequals F_\infty^{U^{p^m}}$ between $\Qp$ and $k_\infty$. Let $G_m^\prime \colonequals \Gal (F_m^\text{cyc}/\Qp)$. Then, we have the projection map $\mathrm{proj}_m:\mathcal{H}_{\widehat{F}_\infty}(\Gamma_p) \to \mathcal{H}_{\widehat{F}_\infty}(G_m^\prime)$. For $z \in H^1_\mathrm{Iw}(k_\infty,T)$, let $z_m$ be its image under the corestriction map $H^1_\mathrm{Iw}(k_\infty,T) \to H^1_\mathrm{Iw}(F_m^\text{cyc},T)$.

\begin{proposition}\label{regulator inverse limit}
Let $z \in H^1_\mathrm{Iw}(k_\infty,T)$. We have $\mathrm{proj}_m \circ \mathcal{L}_{T,k_\infty}(z) = \mathcal{L}_{T,F_m}^{G_m^\prime}(z_m)$ where $\mathcal{L}_{T,F_m}^{G_m^\prime}$ is defined by
$$
\mathcal{L}_{T,F_m}^{G_m^\prime}(x) \colonequals \sum_{\sigma \in \Gal(F_m/\Qp)}[\sigma] \cdot \mathcal{L}_{T,F_m}(\sigma^{-1}\cdot x).
$$
\end{proposition}

\begin{proof}
See \cite[Theorem 4.7]{LZ14}.
\end{proof}

Proposition \ref{regulator inverse limit} tells us that the maps $\mathcal{L}_{T,F_m}^{G_m^\prime}$ form a compatible system along the unramified tower $F_\infty$ and that their inverse limit is given by the two-variable regulator $\mathcal{L}_{T,k_\infty}$. 

For $x \in \mathcal{O}_{F_m}$, define \cite[Definition 3.4]{LZ14}
$$
y_{F_m/\Qp}(x) \colonequals \sum_{\sigma \in \Gal (F_m/\Qp)} \sigma^{-1}\cdot x [\sigma] \in \mathcal{O}_{F_m}[\Gal (F_m/\Qp)].
$$
Write $S_{F_m/\Qp}$ for the image of $y_{F_m/\Qp}$ inside $\mathcal{O}_{F_m}[\Gal (F_m/\Qp)]$ and extend $y_{F_m/\Qp}$ to $\mathcal{O}_{F_m}\otimes \Lambda( \Gamma_{F_m}^\text{cyc})$ by acting only on $\mathcal{O}_{F_m}$.

\begin{definition}
Define the map $\col_{T,F_m}^{G_m^\prime}:H^1_\mathrm{Iw}(F_m^\text{cyc},T) \to \mathcal{O}_{F_m} \otimes_{\Zp} \Lambda(G_m^\prime)^{\oplus d}$ by
$$
\col_{T,F_m}^{G_m^\prime}(z) \colonequals y_{F_m/\Qp}( \col_{T,F_m}(z)) = \sum_{\sigma \in \Gal (F_m/\Qp)} [\sigma] \cdot \col_{T,F_m}(\sigma^{-1}\cdot z).
$$
\end{definition}

For $m\geq 1$, let $\mathrm{proj}_{m/m-1}: \mathcal{H}_{F_m}( G_m^\prime)\to \mathcal{H}_{F_m}(G_{m-1}^\prime)$ be the projection map induced by the surjection $G_m^\prime \twoheadrightarrow G_{m-1}^\prime$. It induces by restriction a map $\mathcal{O}_{F_m}\llbracket G_m^\prime \rrbracket \to \mathcal{O}_{F_m}\llbracket G_{m-1}^\prime \rrbracket$ which we also denote $\mathrm{proj}_{m/m-1}$. It follows from  \cite[Proposition 3.5]{LZ14} that the image of $y_{F_m/\Qp}$ under $\mathrm{proj}_{m/m-1}$ is contained in $\mathcal{O}_{F_{m-1}}\llbracket G_{m-1}^\prime\rrbracket$.

\begin{lemma}\label{Coleman inverse limit}
The Coleman maps $\col_{T,F_m}^{G_m^\prime}$ are compatible in the unramified direction

$$
\begin{tikzcd}
H^1_\mathrm{Iw}(F_m^\mathrm{cyc},T) \arrow{r}{\col_{T,F_m}^{G_m^\prime}} \arrow[swap]{d}{\mathrm{cor}} & \mathcal{O}_{F_m} \llbracket G_m^\prime \rrbracket^{\oplus d} \arrow{d}{\mathrm{proj}_{m/m-1
}} \\
H^1_\mathrm{Iw}(F_{m-1}^\mathrm{cyc},T) \arrow{r}{\col_{T,F_{m-1}}^{G_{m-1}^\prime}} & \mathcal{O}_{F_{m-1}}\llbracket G_{m-1}^\prime \rrbracket^{\oplus d}
\end{tikzcd}
$$
\end{lemma}

\begin{proof}
Let $(z_m) \in \varprojlim_m H^1_{\mathrm{Iw}}(F_m^\text{cyc},T)$. By proposition \ref{regulator inverse limit}, we know that 
\begin{equation}\label{eqn:projection}
\mathrm{proj}_{m/m-1}(\mathcal{L}_{T,F_m}^{G_m^\prime}(z_m)) = \mathcal{L}_{T,F_{m-1}}^{G_{m-1}^\prime}(z_{m-1}).
\end{equation}
Furthermore, the decomposition \eqref{decomposition une variable} implies 
$$
\mathcal{L}_{T,F_m}^{G_m^\prime} =(v_1,\ldots,v_d)\cdot M_T \cdot \col_{T,F_m}^{G_m^\prime}.
$$
Since $M_T$ has entries in $\Lambda(\Gamma^\text{cyc})$, $\mathrm{proj}_{m/m-1}(M_T)=M_T$ for any $m$. By inverting $M_T$ in \eqref{eqn:projection}, we find $\mathrm{proj}_{m/m-1}(\col_{T,F_m}^{G_m^\prime}(z_m))=\col_{T,F_{m-1}}^{G_{m-1}^\prime}(z_{m-1})$. 
\end{proof}

The image of $\mathcal{O}_{F_m}\otimes \Lambda(\Gamma_{F_m}^\text{cyc})$ under $y_{F_m/\Qp}$ is $$S_{F_m/\Qp}\otimes \Lambda(\Gamma_{F_m}^\text{cyc}) \subseteq \mathcal{O}_{F_m}[\Gal (F_m/\Qp)] \otimes \Lambda(\Gamma_{F_m}^\text{cyc})=\mathcal{O}_{F_m}\llbracket G_m^\prime \rrbracket.$$ Taking limits, we find that the inverse limit of the $\col_{T,F_m}^{G_m^\prime}$ takes value in $$S_{F_\infty/\Qp} \widehat{\otimes} \Lambda (\Gal (F_\infty^\text{cyc}/F_\infty))^{\oplus d} \cong \Omega_{\Qp}\cdot \Lambda (U) \widehat{\otimes} \Lambda(\Gamma^\text{cyc})^{\oplus d} \cong \Omega_{\Qp}\cdot \Lambda(\Gamma_p)^{\oplus d}$$
where $S_{F_\infty/\Qp}=\varprojlim_m S_{F_m/\Qp}$ \cite[Equation 3.1]{LZ14}.


\begin{definition}\label{def:two-variable Coleman maps}
Define the two-variable Coleman map by
$$
\col_{T}^{k_\infty}(z) \colonequals \varprojlim_m \sum_{\sigma \in \Gal(F_m/\Qp)}[\sigma] \cdot \col_{T,F_m}(\sigma^{-1}\cdot z_m) : H^1_\mathrm{Iw}(k_\infty,T) \to \Omega_{\mathbf{Q}_p} \cdot \Lambda(\Gamma_p)^{\oplus d}.
$$
\end{definition}

\begin{proposition}\label{coleman inverse limit}
We have the decomposition
$$
\mathcal{L}_{T,k_\infty} = (v_1 \cdots v_d)\cdot M_T \cdot \col_T^{k_\infty}.
$$
\end{proposition}

\begin{proof}
By proposition \ref{regulator inverse limit}, we have
$$
\mathcal{L}_{T,k_\infty}(z) = \varprojlim_m \sum_{\sigma \in \Gal(F_m/\Qp)}[\sigma] \cdot \mathcal{L}_{T,F_m}(\sigma^{-1}\cdot z_m).
$$
Moreover, the decomposition \eqref{decomposition une variable} induces the decomposition
$$
\mathcal{L}_{T,k_\infty}(z) = (v_1 \cdots v_d)\cdot M_T \cdot \varprojlim_m \sum_{\sigma \in \Gal(F_m/\Qp)}[\sigma]\cdot \col_{T,F_m}(\sigma^{-1}\cdot z_m).
$$
\end{proof}

For $1\leq i \leq d$, let $\col_{T,i}^{k_\infty}$ be the $i$th component of the vector $\col_T^{k_\infty}$. By identifying $\Omega_{\mathbf{Q}_p} \cdot \Lambda(\Gamma_p)$ with $\Lambda(\Gamma_p)$, we omit $\Omega_{\mathbf{Q}_p}$ from the notation and see $\text{Col}_{T,i}^{k_\infty}$ as taking value in $\Lambda(\Gamma_p)$.

For simplicity, we also suppose that\bigbreak
\noindent \textbf{(Ram)} $p$ does not divide $h_K$, the class number of $K$.

\begin{lemma}
There is a unique prime above $\gq$ in $K_\infty$.
\end{lemma}

\begin{proof}
Since $\Gamma$ is pro-$p$, \textbf{(Ram)} implies that $K_\infty \cap K(1) = K$. Thus, every prime above $p$ in $K^\text{ac}$ is totally ramified. The result follows since $\gq$ does not split in $K^\text{cyc}$. 
\end{proof}

Since $p$ is split in $K$, the completion of $K$ at $\mathfrak{q}$, denoted $K_\mathfrak{q}$, is isomorphic to $\mathbf{Q}_p$. By \textbf{(Ram)}, there is a unique prime above $\gq$ in $K_\infty$. By abuse of notation, we will also denote by $\mathfrak{p}$ and $\mathfrak{p}^c$ the unique prime above $\gp$ and $\gp^c$ respectively in $K_\infty$. Moreover, we have $K_{\infty,\mathfrak{q}} \cong k_\infty$. Let $\mathbb{D}_{\cris,\mathfrak{q}}(T)$ be the crystalline module of $T$ viewed as a representation of $G_{K_\mathfrak{q}}$. Choose $\{v_{\gp,1},\ldots, v_{\gp,d} \}$ a $\mathbf{Z}_p$-basis of $\mathbb{D}_{\cris,\gp}(T)$ such that $\{v_{\gp,1},\ldots, v_{\gp,d_+} \}$ is a $\mathbf{Z}_p$-basis of $\text{Fil}^0 \mathbb{D}_{\cris,\gp}(T)$ and choose $\{v_{\gp^c,1},\ldots, v_{\gp^c,d} \}$ a $\mathbf{Z}_p$-basis of $\mathbb{D}_{\cris,\gp^c}(T)$ such that $\{v_{\gp^c,1},\ldots, v_{\gp^c,d_+} \}$ is a $\mathbf{Z}_p$-basis of $\text{Fil}^0 \mathbb{D}_{\cris,\gp^c}(T)$. We will also denote by 
\begin{align*}
\mathcal{L}_{T,\gq} : H^1_\text{Iw}(K_{\infty,\gq},T) &\to \Omega_{\mathbf{Q}_p} \cdot \left( \mathcal{H}(\Gamma^\text{cyc})\widehat{\otimes}\Lambda(U) \right) \otimes_{\mathbf{Z}_p} \mathbb{D}_{\cris,\gq}(T) \\
& \cong \mathcal{H}_{\widehat{F}_\infty}(\Gamma_p)\otimes_{\mathbf{Z}_p} \mathbb{D}_{\cris,\gq}(T)
\end{align*}
the local regulator at $\gq$. Let $C_{\varphi,\gq}$ denotes the matrix of the Frobenius operator on $\mathbb{D}_{\cris,\gq}$ and let $M_{T,\gq} = \lim_{n \to \infty}(C_{\varphi,\gq})^{n+1}C_n\cdots C_1$. Then, proposition \ref{coleman inverse limit} induces the decomposition
$$
\mathcal{L}_{T,\gq} = (v_1,\ldots,v_d)M_{T,\gq}\begin{bmatrix} \text{Col}_{T,1}^\gq \\ \vdots \\ \text{Col}_{T,d}^\gq\end{bmatrix}
$$
where $\text{Col}_{T,i}^\gq:H^1_{\mathrm{Iw}}(K_{\infty,\gq},T) \to \Lambda(\Gamma_p)$ is the $v_{\gq,i}$-component of the Coleman map given by definition \ref{def:two-variable Coleman maps}.
Let $\{v_{\gq,1}^\prime,\ldots,v_{\gq,d}^\prime \}$ denote the dual basis of $\mathbb{D}_{\cris,\gq}(T^\ast (1))$. Denote by $C_{\varphi,\gq}^\ast$ the matrix of the Frobenius on $\mathbb{D}_{\cris,\gq}(T^\ast (1))$. Since $p^{-1}\varphi$ is the adjoint of the Frobenius acting on $\mathbb{D}_{\cris,\gq} (T^\ast (1))$ by \eqref{dual operator} and that we have chosen a dual basis, we get
$$
C_{\varphi,\gq}^\ast = \frac{1}{p}\cdot (C_{\varphi,\gq}^{-1})^t.
$$
As we did for $T$, we define 
$$
M_{T^\ast (1),\gq}=\lim_{n\to \infty} (C_{\varphi,\gq}^\ast)^{n+1}C^\ast_{n,\gq} \cdots C^\ast_{1,\gq}
$$
where 
\begin{equation}\label{matrice duale}
C^\ast_{n,\gq} = \left[
\begin{array}{c|c}
\Phi_{p^n}(1+X) I_{d_-} & 0 \\
\hline
0 & I_{d_+}
\end{array}
\right] C_\gq^t.
\end{equation}
We get a similar decomposition for $\mathcal{L}_{T^\ast (1),\gq}$:
\begin{equation}\label{regulator decomposition}
\mathcal{L}_{T^\ast (1),\gq} = (v_{\gq,1}^\prime,\ldots, v_{\gq,d}^\prime) M_{T^\ast (1),\gq}\begin{bmatrix} \text{Col}_{T^\ast (1),1}^\gq \\ \vdots \\ \text{Col}_{T^\ast (1),d}^\gq\end{bmatrix}
\end{equation}
where $\text{Col}_{T^\ast (1),i}^\gq:H^1_{\mathrm{Iw}}(K_{\infty,\gq},T^\ast (1)) \to \Lambda(\Gamma_p)$ is the $v^\prime_{\gq,i}$-component of the Coleman map given by definition \ref{def:two-variable Coleman maps} for $T^\ast (1)$. 
For $I_\gq \subseteq \{1,\ldots, d\}$, we let $\text{Col}_{T,I_\gq}^\gq := \bigoplus_{i \in I_\gq} \text{Col}_{T,i}^\gq$. For a fixed $I_\gq \subseteq \{1,\ldots, d\}$ with $\#I_\gq=d_+$, we define $H^1_{I_\gq}(K_{\infty,\gq},T^\dagger)$ to be the orthogonal complement of $\ker \text{Col}_{T,I_\gq}^\gq$ under local Tate duality
$$
H^1_\text{Iw}(K_{\infty,\gq},T) \times H^1(K_{\infty,\gq},T^\dagger) \to \mathbf{Q}_p /\mathbf{Z}_p.
$$
Let $\Sigma$ be a finite set of places of $K$ containing the prime above $p$, the archimedean places and the prime of ramification of $T^\dagger$. Let $K_\Sigma$ be the maximal extension of $K$ unramified outside $\Sigma$. Let $\Sigma^\prime$ the set of places of $K_\infty$ lying above the places in $\Sigma$. If $M$ is any $\Gal(K_\Sigma /K_\infty)$-module, we denote $H^1(K_\Sigma/K_\infty,M)$ by $H^1_\Sigma(K_\infty,M)$. Let $\underline{I}\colonequals (I_\gp,I_{\gp^c})$ a choice of subsets as above. Let
$$
\mathcal{P}_{\Sigma,\underline{I}}(T^\dagger/K_\infty) \colonequals \bigoplus_{v \in \Sigma^\prime, v \nmid p} \frac{H^1(K_{\infty,v},T^\dagger)}{H^1_f(K_{\infty,v},T^\dagger)} \bigoplus_{\gq \in \{\gp,\gp^c\}} \frac{H^1(K_{\infty,\gq},T^\dagger)}{H^1_{I_\gq}(K_{\infty,\gq},T^\dagger)}
$$
where $H^1_f(K_{\infty,v},T^\dagger)$ is the unramified subgroup of $H^1(K_{\infty,v},T^\dagger)$. Then, $\Sel_{\underline{I}}(T^\dagger / K_\infty)$ is defined to be 
$$
\ker \left( H^1_\Sigma(K_\infty, T^\dagger) \to \mathcal{P}_{\Sigma,\underline{I}}(T^\dagger/K_\infty) \right).
$$
We define $\text{Col}_{T^\ast (1),I_\gq}^\gq$, $H_I^1(K_{\infty,\gq},T^\ast (1))$, $H_f^1(K_{\infty,v},T^\ast (1)^\dagger)$ and $\Sel_{\underline{I}}(T^\ast (1)^\dagger /K_\infty)$ in the obvious way.

\subsection{Orthogonality of local conditions}\label{sec:orthogonality}

Next, we show that the local conditions at $p$ defining $\Sel_{\underline{I}}(T^\dagger / K_\infty)$ and $\Sel_{\underline{I}}(T^\ast (1)^\dagger /K_\infty)$ are the orthogonal complement of each other under Perrin-Riou's pairing.



\begin{lemma}\label{rank}
For any $I_\gq \subseteq \{1,\ldots,d\}$, the $\Lambda(\Gamma_p)$-module $\ker (\text{Col}_{T,I_\gq}^\gq)$ is 
of rank $d-\#I_\gq$.
\end{lemma}

\begin{proof}
By \cite[Theorem A.2]{LZ14}, $H^1_\text{Iw}(K_{\infty,\gq},T)$ is a finitely generated $\Lambda(\Gamma_p)$-module of rank $d$. Let $\col_{T,\Qp,i}$ be the $i$th component of the one-variable Coleman map $\col_{T,\Qp}$. By \cite[Corollary 2.22]{BL17}, $\text{Im}(\col_{T,\Qp,I_{\mathfrak{q}}})$ is contained in a free $\Lambda(\Gamma^\text{cyc})$-module of rank $d-\#I_\mathfrak{q}$, with finite index. By the definition of the regulator map \cite[Definition 4.6]{LZ14},
$$
\im (\mathcal{L}_{T,\gq}) = \im (\mathcal{L}_{T,\Qp}) \widehat{\otimes}_{\Zp} S_{F_\infty/\Qp}.
$$
By the decomposition of proposition \ref{coleman inverse limit} and of \eqref{decomposition une variable}, 
$$
(v_1 \cdots v_d)\cdot M_T \cdot \im (\col_T^\gq) = (v_1 \cdots v_d)\cdot M_T \cdot \im (\col_{T,\Qp}) \widehat{\otimes}_{\Zp} S_{F_\infty/\Qp}.
$$
After inverting $M_T$, we get that $\im (\col_{T,I_\mathfrak{q}}^\gq ) = \im (\col_{T,\Qp,I_\mathfrak{q}})\widehat{\otimes}_{\Zp} S_{F_\infty/\Qp}$ is contained in a free $\Lambda(\Gamma_p)$-module of rank $d-\#I_\mathfrak{q}$, with finite index. It follows that $\ker (\col_{T,I_\gq}^\gq)$ is of rank $d-\# I_\gq$ over $\Lambda(\Gamma_p)$.
\end{proof}

Let $\langle \sim,\sim \rangle_n$ be the local Tate pairing $H^1(K_{n,\gq},T)\times H^1(K_{n,\gq},T^\ast (1)) \to \mathbf{Z}_p$. Consider Perrin-Riou's pairing 
\begin{align*}
\langle \sim,\sim \rangle : H^1_\text{Iw}(K_{\infty,\gq},T) \times H^1_\text{Iw}(K_{\infty,\gq},T^\ast (1)) &\to \Lambda(\Gamma_p) \\
( (x_n),(y_n) ) &\mapsto \varprojlim_n \sum_{\sigma \in \Gamma_{p,n}} \langle x_n, y_n^\sigma \rangle_n\cdot \sigma \in \varprojlim_n \mathbf{Z}_p[\Gamma_{p,n}]
\end{align*}
and the crystalline pairing
$$
[\sim,\sim]: \mathbb{D}_{\cris,\gq}(T) \times \mathbb{D}_{\cris,\gq}(T^\ast (1)) \to \mathbf{Z}_p.
$$
A version of Perrin-Riou's explicit reciprocity law for $\mathbf{Z}_p^2$-extensions was proved by Loeffler--Zerbes \cite[Theorem 4.17]{LZ14}. If $x \in H^1_\text{Iw}(K_{\infty,\gq},T)$ and $y \in H^1_\text{Iw}(K_{\infty,\gq},T^\ast (1))$, then 
\begin{equation}\label{reciprocity}
[\mathcal{L}_{T,\gq}(x),\mathcal{L}_{T^\ast (1),\gq}(y)] = -\sigma_{-1}\cdot \ell_0 \cdot \langle x,y \rangle
\end{equation}
where $\ell_0 = \frac{\log \gamma}{\log \chi (\gamma)}$ for any $\gamma \in \Gamma^\text{cyc}$. The element $\sigma_{-1}$ is the unique element of the inertia subgroup of $\Gamma_p$ such that $\chi(\sigma_{-1})=-1$.

\begin{lemma}\label{matrices product}
Let $x \in H^1_\text{Iw}(K_{\infty,\gq},T)$ and $y \in H^1_\text{Iw}(K_{\infty,\gq},T^\ast (1))$, then
$$
[\mathcal{L}_{T,\gq}(x),\mathcal{L}_{T^\ast (1),\gq}(y)] = \frac{\log (1+X)}{pX}\cdot \text{Col}_T^\gq(x)^t \cdot \text{Col}_{T^\ast (1)}^\gq(y).
$$
\end{lemma}

\begin{proof}
The proof is the same as \cite[Lemma 3.1]{LP17}.
\end{proof}

\begin{lemma}\label{orthogonal}
Let $I_\gq \subseteq \{1,\ldots,d\}$ and $I_\gq^c$ its complement. Then $\ker \text{Col}_{T^\ast (1), I_\gq^c}^\gq$ is the orthogonal complement of $\ker \text{Col}_{T,I_\gq}^\gq$ with respect to the pairing $\langle \sim,\sim \rangle$.
\end{lemma}

\begin{proof}
Again, we can follow the exact same arguments as in \cite[Lemma 3.2]{LP17}. Let $x \in H^1_\text{Iw}(K_{\infty,\gq},T)$ and $y \in H^1_\text{Iw}(K_{\infty,\gq},T^\ast (1))$. By the reciprocity law \eqref{reciprocity} and lemma \ref{matrices product}, $\langle x,y\rangle =0$ if and only if $[\mathcal{L}_{T,\gq}(x),\mathcal{L}_{T^\ast (1),\gq}(y)]=0$ which happens if and only if $\text{Col}_T^\gq(x)^t \cdot \text{Col}_{T^\ast (1)}^\gq(y)=0$. Thus, if $x \in \ker \text{Col}_{T,I_\gq}^\gq$, 
\begin{equation}\label{reci}
\langle x,y\rangle = 0 \Leftrightarrow \sum_{k \not\in I_\gq} \text{Col}_{T,k}^\gq(x) \cdot \text{Col}_{T^\ast (1),k}^\gq(y) =0.
\end{equation}
So, $\ker \text{Col}_{T^\ast (1),I_\gq^c}^\gq \subseteq \left( \ker \text{Col}_{T,I_\gq}^\gq \right)^\perp$. Lemma \ref{rank} implies that for all $k \in \{1,\ldots,d\}$, there exists $x_k$ such that
$$
\text{Col}_{T,j}^\gq(x_k) \begin{cases} \text{$=0$ if $j\in \{1,\ldots,d\}\setminus \{k\}$,} \\ \text{$\neq 0$ if $j=k$.} \end{cases}
$$
In particular, if $k \not\in I_\gq$, then such $x_k\in \ker \text{Col}_{T,I_\gq}^\gq$. If $y \in \left( \ker \text{Col}_{T,I_\gq}^\gq \right)^\perp$, then $\langle x_k,y\rangle =0$. Therefore, \eqref{reci} tells us that $\text{Col}_{T^\ast (1),k}^\gq(y)=0$. Since this is true for all $k \in I_\gq^c$, we have $y \in \ker \text{Col}_{T^\ast (1),I_\gq^c}^\gq$ as required.
\end{proof}

\subsection{Strict Selmer groups}

Fix an indexing set $I_\gq \subseteq \{1,\ldots,d\}$. Our next goal is to describe the local condition $H^1_{I_\gq}(K_{\infty,\gq},T^\dagger)$ at the level of $K_{n,\gq}$. Let us make the following assumption:\bigbreak
\noindent \textbf{(Tors)} The Galois cohomology groups $H^0(K_\gq, T/pT)$ and $H^2(K_\gq, T/pT)$ are trivial.

\begin{remark}
Since $T/pT \cong T^\ast(1)^\dagger [p]$, $H^0(K_\gq, T/pT)=0$ is equivalent to $$H^0(K_\gq,T^\ast(1)^\dagger [p])=0.$$ Furthermore, by local Tate duality, $H^2(K_\gq, T/pT)$ is trivial if and only if $H^0(K_\gq, T^\dagger [p])$ is trivial.
\end{remark}
Since $\Gal( K_{\infty,\gq}/K_\gq)$ is a pro-$p$ group, \textbf{(Tors)} combined with the orbit-stabilizer theorem gives that $H^0(K_{\infty,\gq},T^\dagger)=0$. Then, by the inflation-restriction exact sequence, $H^1(K_{\infty,\gq},T^\dagger)^{\Gamma_{p,n}}\cong H^1(K_{n,\gq},T^\dagger)$.
Consider the short exact sequence
$$
0 \to T^\ast(1) \to \mathcal{M}_p^\ast(1) \to T^\dagger \to 0
$$
where the first arrow is the inclusion map and the second arrow is the projection $\mathcal{M}_p^\ast(1) \to \mathcal{M}_p^\ast(1)/T^\ast(1) =T^\ast(1)\otimes \Qp/\Zp \cong T^\dagger$. Note that if $T$ satisfies \textbf{(Tors)}, its Tate dual $T^\ast (1)$ also satisfies it. Thus, we get the attached short exact sequence in cohomology
$$
0\to H^1(K_{n,\gq},T^\ast(1)) \xrightarrow{i_n} H^1(K_{n,\gq},\mathcal{M}_p^\ast(1)) \xrightarrow{\pi_n} H^1(K_{n,\gq},T^\dagger) \to 0.
$$
Define $H^1_{I_\gq}(K_{n,\gq},T^\dagger) \colonequals H^1_{I_\gq}(K_{\infty,\gq},T^\dagger)^{\Gamma_{p,n}} \subseteq H^1(K_{n,\gq},T^\dagger)$. For $n\geq 0$, define the submodule $(\ker \text{Col}_{T^\ast(1),I_\gq^c}^\gq)_n$ to be the image of $\ker \text{Col}_{T^\ast(1),I_\gq^c}^\gq$ under the natural map $$H^1_\text{Iw}(K_{\infty,\gq},T^\ast(1)) \to H^1(K_{n,\gq},T^\ast(1)).$$ The image of $(\ker \text{Col}_{T^\ast(1),I_\gq^c}^\gq)_n$ under the map $i_n$
generates a $\mathbf{Q}_p$-vector space inside $H^1(K_{n,\gq},\mathcal{M}_p^\ast(1))$. We denote its image under $\pi_n$ by $\overline{(\ker \text{Col}_{T^\ast(1),I_\gq^c}^\gq)_n}$. Let $(\ker \text{Col}_{T,I_\gq}^\gq)_n$ and $\overline{(\ker \text{Col}_{T,I_\gq}^\gq)_n}$ be defined in the same way by using the exact sequence
$$
0\to H^1(K_{n,\gq},T) \to H^1(K_{n,\gq},\mathcal{M}_p) \to H^1(K_{n,\gq},T^\ast(1)^\dagger) \to 0
$$
instead.
\begin{lemma}\label{ortho col}
For $n\geq 0$, $\overline{(\ker \text{Col}_{T^\ast (1),I_\gq^c}^\gq)_n}$ is the orthogonal complement of $(\ker \text{Col}_{T,I_\gq}^\gq)_n$ under Tate's local pairing
$$
H^1(K_{n,\gq},T) \times H^1(K_{n,\gq}, T^\ast (1)^\dagger) \to \mathbf{Q}_p/\mathbf{Z}_p.
$$
Furthermore, We have $$\overline{(\ker \text{Col}_{T^\ast(1),I_\gq^c}^\gq)_n} = H^1_{I_\gq}(K_{n,\gq},T^\dagger)\quad \text{and}\quad \overline{(\ker \text{Col}_{T,I_\gq}^\gq)_n} = H^1_{I_\gq}(K_{n,\gq},T^\ast(1)^\dagger).$$
\end{lemma}

\begin{proof}
Since the orthogonality conditions are satisfied by lemma \ref{orthogonal}, the arguments used in the proof of \cite[Lemma 1.6]{Pon20} for one-variable Coleman maps can be applied without modification to our two-variable Coleman maps.
\end{proof}
 
Let $\Sigma_n$ be the set of places above those in $\Sigma$ in $K_n$. Let
$$
\mathcal{P}_{\Sigma_n,\underline{I}}(T^\dagger/K_n) \colonequals \bigoplus_{v \in \Sigma_n, v \nmid p} \frac{H^1(K_{n,v},T^\dagger)}{H^1_f(K_{n,v},T^\dagger)} \bigoplus_{\gq \in \{\gp,\gp^c\}} \frac{H^1(K_{n,\gq},T^\dagger)}{H^1_{I_\gq}(K_{n,\gq},T^\dagger)}
$$
where $H^1_f(K_{n,v},T^\dagger)$ is again the unramified subgroup of $H^1(K_{n,v},T^\dagger)$. We define the strict multi-signed Selmer group at level $n$ by using the local condition $H^1_{I_\gq}(K_{n,\gq},T^\dagger)$,
$$
\Sel_{\underline{I}}^\text{str}(T^\dagger /K_{n}) := \ker \left( H^1_{\Sigma_n}(K_n, T^\dagger) \to \mathcal{P}_{\Sigma_n,\underline{I}}(T^\dagger/K_n) \right).
$$
We define 
$$\text{Sel}_{\underline{I}}^\text{str}(T^\dagger/K_{\infty}) \colonequals \varinjlim_n \text{Sel}_{\underline{I}}^\text{str}(T^\dagger /K_{n})$$ 
where the transition maps are those induced by restriction on cohomology.

\begin{remark}
Here, we mean strict Selmer groups in the sense of Greenberg \cite{Gre89} and not as what is sometime called fine Selmer groups.
\end{remark}

\begin{proposition}\label{strict}
Suppose that \textbf{(Tors)} is valid. The two $\Lambda(\Gamma)$-modules $\Sel_{\underline{I}}(T^\dagger / K_\infty)$ and $\Sel_{\underline{I}}^\text{str}(T^\dagger/K_{\infty})$ are isomorphic.
\end{proposition}

\begin{proof}
Since $H^1_{I_\gq}(K_{\infty,\gq},T^\dagger)$ is a discrete $\Gamma_p$-module, \cite[Proposition 1.1.8]{Neu13} gives
$$
H^1_{I_\gq}(K_{\infty,\gq},T^\dagger) = \varinjlim_n H^1_{I_\gq}(K_{\infty,\gq},T^\dagger)^{\Gamma_{p,n}}.
$$
By the exactness of $\varinjlim_n$, 
$$
\varinjlim_n \frac{H^1(K_{n,\gq},T^\dagger)}{H^1_{I_\gq}(K_{n,\gq},T^\dagger)} \cong \frac{\varinjlim_n H^1(K_{n,\gq},T^\dagger)}{\varinjlim_n H^1_{I_\gq}(K_{n,\gq},T^\dagger)} = \frac{H^1(K_{\infty,\gq},T^\dagger)}{H^1_{I_\gq}(K_{\infty,\gq},T^\dagger)}.
$$
Also, $\varinjlim_n H^1_f(K_{n,v},T^\dagger) = H^1_f(K_{\infty,v},T^\dagger)$ by definition. Thus, the two Selmer groups $\Sel_{\underline{I}}(T^\dagger / K_\infty)$ and $\Sel_{\underline{I}}^\text{str}(T^\dagger/K_{\infty})$ are identified.
\end{proof}
We define $\Sel_{\underline{I}}^\text{str}((T^\ast(1))^\dagger/K_{n})$ with the local condition $H^1_{I_\gq}(K_{n,\gq},(T^\ast (1))^\dagger)$ and let $$\Sel_{\underline{I}}^\text{str}((T^\ast (1))^\dagger/K_{\infty}):= \varinjlim_n \Sel_{\underline{I}}^\text{str}((T^\ast (1))^\dagger/K_{n}).$$ By replacing $T$ with $T^\ast (1)$ in proposition \ref{strict}, we get $$\Sel_{\underline{I}}^\text{str}((T^\ast(1))^\dagger/K_{\infty}) \cong \Sel_{\underline{I}}((T^\ast (1))^\dagger/K_{\infty}).$$

\subsection{Flach's pairing}
In \cite{Fla90}, Flach gives a generalization of the Cassel--Tate pairing for motives over number fields. For $A$ an abelian group, denote by $A_{/\text{div}}$ the quotient of $A$ by its maximal divisible subgroup. The main result of Flach implies the following:

\begin{proposition}
Suppose that \textbf{(Tors)} is valid. There is a perfect pairing 
$$
\text{Sel}_{\underline{I}}^{\text{str}}(T^\dagger/K_{n})_{/\text{div}} \times \text{Sel}_{\underline{I}^c}^{\text{str}}((T^\ast (1))^\dagger/K_{n})_{/\text{div}} \to \mathbf{Q}_p/\mathbf{Z}_p.
$$
\end{proposition}

\begin{proof}
For $v$ a place above $p$ in $K_n$, lemma \ref{ortho col} shows that the local condition defining $\Sel_{\underline{I}}^{\text{str}}(T^\dagger/K_n)$ is the orthogonal complement of the one defining $\Sel_{\underline{I}^c}^{\text{str}}(T^\ast (1)^\dagger/K_n)$. For $v \nmid p$,  \cite[Proposition 3.8]{BK90} shows that $H^1_f(K_{n,v},T^\dagger)^\perp = H^1_f(K_{n,v},T^\ast (1)^\dagger)$. So we can apply \cite[Theorem 1]{Fla90}.
\end{proof}

\section{\texorpdfstring{$\Gamma$-systems}{Gamma-systems}}\label{gamma-systems}

We recall the theory of $\Gamma$-systems from \cite{LLTT}. For this section only, $\Gamma$ will be any abelian $p$-adic Lie group isomorphic to $\mathbf{Z}_p^d$ for some $d \geq 1$. Let $\Lambda(\Gamma) = \mathbf{Z}_p \llbracket T_1,\ldots,T_d \rrbracket$ be the associated Iwasawa algebra. Put $\Gamma_n := \Gamma /\Gamma^{p^n}$. Consider a collection
$$
\mathfrak{A} = \{ \mathfrak{a}_n, \mathfrak{b}_n, \langle \sim,\sim \rangle_n ,r_m^n, c_m^n : n,m \in \mathbf{Z}_{\geq 0}, n \geq m \}
$$
where \\
($\Gamma$-1) $\mathfrak{a}_n$ and $\mathfrak{b}_n$ are finite abelian groups with an action of $\Lambda(\Gamma)$ factoring through $\mathbf{Z}_p[\Gamma_n]$. \\
($\Gamma$-2) For $n \geq m$,
$$
r_m^n: \mathfrak{a}_m \times \mathfrak{b}_m \to \mathfrak{a}_n \times \mathfrak{b}_n,
$$
$$
c_m^n: \mathfrak{a}_n \times \mathfrak{b}_n \to \mathfrak{a}_m \times \mathfrak{b}_m
$$
are $\Gamma$-morphisms such that $r_m^n(\mathfrak{a}_m) \subseteq \mathfrak{a}_n$, $r_m^n(\mathfrak{b}_m) \subseteq \mathfrak{b}_n$, $c_m^n(\mathfrak{a}_n) \subseteq \mathfrak{a}_m$, $c_m^n(\mathfrak{b}_n) \subseteq \mathfrak{b}_m$ and $c_n^n=r_n^n=\text{id}$. Also, $\{ \mathfrak{a}_n \times \mathfrak{b}_n, r_m^n \}_n$ form an inductive system and $\{ \mathfrak{a}_n \times \mathfrak{b}_n, c_m^n \}_n$ form a projective system. \\
($\Gamma$-3) We have 
$$r_m^n \circ c_m^n = N_{\Gamma_n /\Gamma_m}: \mathfrak{a}_n \times \mathfrak{b}_n \to \mathfrak{a}_n \times \mathfrak{b}_n
$$
(where $N_{\Gamma_n /\Gamma_m}:=\sum_{\sigma \in \ker (\Gamma_n \to \Gamma_m)}\sigma$ is the norm associated with $\Gamma_n \twoheadrightarrow \Gamma_m$) and 
$$
c_m^n \circ r_m^n = p^{d(n-m)}\cdot \text{id} : \mathfrak{a}_m \times \mathfrak{b}_m \to \mathfrak{a}_m \times \mathfrak{b}_m.
$$
($\Gamma$-4) For each $n$, $\langle \sim, \sim \rangle_n : \mathfrak{a}_n \times \mathfrak{b}_n \to \mathbf{Q}_p / \mathbf{Z}_p$ is a perfect pairing respecting $\Gamma$-action as well as the morphisms $c_m^n$ and $r_m^n$ in the sense that
$$
\langle \gamma\cdot a , \gamma \cdot b \rangle_n = \langle a,b \rangle_n \ \forall \gamma \in \Gamma,
$$
$$
\langle a, r_m^n(b) \rangle_n = \langle c_m^n (a),b \rangle_m
$$
and
$$
\langle r_m^n (a),b \rangle_n = \langle a,c_m^n(b) \rangle_m.
$$
Write $\mathfrak{a} := \varprojlim_{n}\mathfrak{a}_n$ and $\mathfrak{b}:=\varprojlim_n \mathfrak{b}_n$. If condition ($\Gamma$-1) through ($\Gamma$-4) hold and $\mathfrak{a}$ and $\mathfrak{b}$ are both finitely generated torsion $\Lambda(\Gamma)$-module, we call $\mathfrak{A}$ a $\Gamma$-system. \bigbreak

Suppose that $N$ is a finitely generated torsion $\Lambda(\Gamma)$-module. There is a pseudo-isomorphism
$$
N \to \bigoplus_{i=1}^m \Lambda(\Gamma)/\xi_i^{r_i} \Lambda(\Gamma)
$$
where each $\xi_i$ is irreducible and $r_i$ are non-negative integers. The characteristic ideal of $N$ is $\chi(N)\colonequals \prod_{i=1}^m \xi_i^{r_i}$ and we let 
$$
[N] \colonequals \bigoplus_{i=1}^m \Lambda(\Gamma)/\xi_i^{r_i} \Lambda(\Gamma).
$$
We say that $f \in \Lambda(\Gamma)$ is a simple element if there exists $\gamma \in \Gamma \setminus \Gamma^{p}$ and $\zeta \in \mu_{p^\infty}$ such that $f=f_{\gamma,\zeta}$, where
$$
f_{\gamma,\zeta}\colonequals \prod_{\sigma \in \text{Gal}(\mathbf{Q}_p(\zeta)/\mathbf{Q}_p)} (\gamma - \sigma(\zeta)).
$$
Define $[N]_\text{si}$ as the sum over the $\xi_i$ which are simple element and $[N]_\text{ns}$ as its complement. We get a decomposition $[N]=[N]_\text{si} \oplus [N]_\text{ns}$. Let $\iota$ be the involution on $\Lambda(\Gamma)$ induced by the map $\sigma \mapsto \sigma^{-1}$ on $\Gamma$. Let $\lambda \in \Lambda(\Gamma)$ and $x \in N$. The $\Lambda(\Gamma)$-module with the same underlying group as $N$ but with twisted action $\lambda \cdot x = \iota(\lambda) x$ will be denoted by $N^\iota$. 

\begin{lemma}\label{involution simple}
We have $[N]_\text{si}^\iota = [N]_\text{si}$.
\end{lemma}

\begin{proof}
This is \cite[Section 2.2 equation (9)]{LLTT}.
\end{proof}

To get a similar result for the non-simple part, one need to build up a new system from the $\Gamma$-system $\mathfrak{A}$. Denote by $r_n$ the natural morphism $\mathfrak{a}_n \to \varinjlim_m \mathfrak{a}_m$ and its kernel by $\mathfrak{a}_n^0$. The module $\mathfrak{b}_n^0$ is defined similarly. Write $\mathfrak{a}_n^1$ (resp. $\mathfrak{b}_n^1$) for the annihilator of $\mathfrak{b}_n^0$ (resp. $\mathfrak{a}_n^0$) with respect to the perfect pairing $\langle \sim,\sim \rangle_n$. Define $\mathfrak{a}_n^\prime$ to be the image of $\mathfrak{a}_n^1$ under the quotient map $\mathfrak{a}_n \to \mathfrak{a}_n/\mathfrak{a}_n^0$. The module $\mathfrak{b}_n^\prime$ is defined similarly. Let $\mathfrak{a}^i \colonequals \varprojlim_n \mathfrak{a}_n^i$, $\mathfrak{b}^i \colonequals \varprojlim_n \mathfrak{b}_n^i$ ($i=0,1$), $\mathfrak{a}^\prime \colonequals \varprojlim_n \mathfrak{a}_n^\prime$ and $\mathfrak{b}^\prime \colonequals \varprojlim_n \mathfrak{b}_n^\prime$ where the transition maps are induced by $c_m^n$.

\begin{lemma}\label{prime}
The following statements are valid. \bigbreak
(a) We have isomorphisms $\varinjlim_n \mathfrak{b}_n/\mathfrak{b}_n^0 \cong (\mathfrak{a}^1)^\vee$ and $\varinjlim_n \mathfrak{a}_n/\mathfrak{a}_n^0 \cong (\mathfrak{b}^1)^\vee$ where $()^\vee$ is the Pontryagin dual. \bigbreak
(b) There are short exact sequence of $\Lambda(\Gamma)$-modules
$$
0\to \mathfrak{a}_n^0 \to \mathfrak{a} \to \mathfrak{a}^\prime \to 0,
$$
$$
0\to \mathfrak{b}_n^0 \to \mathfrak{b} \to \mathfrak{b}^\prime \to 0.
$$
\end{lemma}

\begin{proof}
See \cite[Lemma 4.4]{AL21}.
\end{proof}

The upshot of this new system is that it is well behaved under the involution $\iota$.
\begin{proposition}\label{non simple}
Let $\mathfrak{A}$ be a $\Gamma$-system. Then we have 
$$
[\mathfrak{a}^\prime]^\iota_\text{ns} = [\mathfrak{b}^\prime]_\text{ns}.
$$
\end{proposition}

\begin{proof}
See \cite[Corollary 3.3.4]{LLTT}.
\end{proof}

\section{Algebraic functional equation}

Let $X_{\underline{I}}(T^\dagger/K_\infty)$ (resp. $X_{\underline{I}^c}(T^\ast(1)^\dagger/K_\infty)$) be the Pontryagin dual of $\Sel_{\underline{I}}(T^\dagger/K_\infty)$ (resp. $\Sel_{\underline{I}^c}(T^\ast(1)^\dagger/K_\infty)$). To prove the functional equation $X_{\underline{I}}(T^\dagger/K_\infty)^\iota \sim X_{\underline{I}^c}(T^\ast(1)^\dagger/K_\infty)$, it suffices to show that $[X_{\underline{I}}(T^\dagger/K_\infty)]_\text{si}^\iota=[X_{\underline{I}^c}(T^\ast(1)^\dagger/K_\infty)]_\text{si}$ and $[X_{\underline{I}}(T^\dagger/K_\infty)]_\text{ns}^\iota=[X_{\underline{I}^c}(T^\ast(1)^\dagger/K_\infty)]_\text{ns}$. The assertion about non-simple parts will follow by using techniques from \cite{AL21} and \cite{LLTT}. However, for simple parts, we will need a careful analysis of bases of Dieudonné modules. 

\subsection{Non-simple parts}\label{sct:Non-simple}

We go back to the notation $\Gamma=\Gal(K_\infty/K)$. We start by defining the $\Gamma$-system underlying the argument. Let $\mathfrak{a}_n = \text{Sel}_{\underline{I}}^{\text{str}}(T^\dagger/K_{n})_{/\text{div}}$ and $\mathfrak{b}_n = \text{Sel}_{\underline{I}^c}^{\text{str}}((T^\ast (1))^\dagger/K_{n})_{/\text{div}}$. Let $r_m^n$ be the restriction maps on the multi-signed Selmer group induced by the restriction maps on cohomology. Let $c_m^n$ be the corestriction maps on the multi-signed Selmer group induced by the corestriction maps on cohomology. Let $\langle \sim,\sim \rangle_n$ be Flach's pairing. Note that both $X_{\underline{I}}(T^\dagger /K_\infty)$ and $X_{\underline{I}^c}(T^\ast (1)^\dagger /K_\infty)$ are finitely generated over $\Lambda(\Gamma)$ since $H^1_{\mathrm{Iw},\Sigma}(K_\infty,T^\dagger)$ is \cite[Theorem A.4]{LZ14}.

\begin{lemma}\label{G-system}
Suppose that $X_{\underline{I}}(T^\dagger /K_\infty)$ and $X_{\underline{I}^c}(T^\ast (1)^\dagger /K_\infty)$ are torsion over $\Lambda(\Gamma)$. Then
$$
\mathfrak{A} = \{ \mathfrak{a}_n, \mathfrak{b}_n, \langle \sim,\sim \rangle_n ,r_m^n, c_m^n : n,m \in \mathbf{Z}_{\geq 0}, n \geq m \}
$$
is a $\Gamma$-system.
\end{lemma}

\begin{proof}
One can check that the condition ($\Gamma$-1)-($\Gamma$-4) are satisfied (for ($\Gamma$-1), notice that $H^1_{\Sigma_n}(K_n,T^\dagger)^\vee$ is a finitely generated $\Zp$-module, thus $\Sel_{\underline{I}}^\text{str}(T^\dagger/K_n)$ is of the form $(\Qp/\Zp)^s \times P$ where $s\geq 0$ and $P$ is a finite group). Write $\Sel_{\underline{I}}^{\text{str}}(T^\dagger/K_{n})_{\text{div}}$ for the divisible part of $\Sel_{\underline{I}}^{\text{str}}(T^\dagger/K_{n})$ and $\Sel_{\underline{I}}^{\text{str}}(T^\dagger/K_{\infty})_{\text{div}}$ for the direct limit $$\varinjlim_n \Sel_{\underline{I}}^{\text{str}}(T^\dagger/K_{n})_{\text{div}}.$$ By definition, the module $\varinjlim_n \mathfrak{a}_n$ fits into the short exact sequence
$$
0\to \Sel_{\underline{I}}^{\text{str}}(T^\dagger/K_{\infty})_{\text{div}} \to \Sel_{\underline{I}}(T^\dagger/K_{\infty})\to  \varinjlim_n \mathfrak{a}_n \to 0,
$$
where we have identified $\Sel_{\underline{I}}^{\text{str}}(T^\dagger/K_{\infty})$ with $\Sel_{\underline{I}}(T^\dagger/K_{\infty})$ via proposition \ref{strict}. Upon taking Pontryagin duals, we get the short exact sequence
\begin{equation} \label{main es}
0 \to (\varinjlim_n \mathfrak{a}_n)^\vee \to X_{\underline{I}}(T^\dagger /K_\infty) \to \Sel_{\underline{I}}^{\text{str}}(T^\dagger/K_{\infty})_{\text{div}}^\vee \to 0.
\end{equation}
But, by property ($\Gamma$-4), $(\varinjlim_n \mathfrak{a}_n)^\vee \cong \varprojlim_n \mathfrak{b}_n = \mathfrak{b}$. Thus, the submodule $\mathfrak{b}$ of the $\Lambda(\Gamma)$-torsion module $X_{\underline{I}}(T^\dagger /K_\infty)$ is also a torsion $\Lambda(\Gamma)$-module. We can show that $\mathfrak{a}$ is torsion in a similar way.
\end{proof}

\begin{remark}
See \cite[Remark 4.3]{DR21} for examples where the torsion hypothesis of lemma \ref{G-system} are satisfied.
\end{remark}

Put $Y_{\underline{I}}(T^\dagger/K_\infty)\colonequals \Sel_{\underline{I}}^{\text{str}}(T^\dagger/K_{\infty})_{\text{div}}^\vee$ and $Y_{\underline{I}^c}(T^\ast (1)^\dagger/K_\infty)\colonequals \Sel_{\underline{I}^c}^{\text{str}}(T^\ast (1)^\dagger/K_{\infty})_{\text{div}}^\vee$.

\begin{lemma}\label{restriction}
For every $n$, the restriction map
$$
\Sel_{\underline{I}}^\text{str}(T^\dagger /K_{n}) \to \Sel_{\underline{I}}(T^\dagger /K_\infty)^{\Gamma_n}
$$
is an injection.
\end{lemma}

\begin{proof}
For each place $v$ of $K_\infty$, let $J_v$ denote the local condition $$H^1(K_{\infty,v},T^\dagger)/H^1_f(K_{\infty,v},T^\dagger)$$ when $v\in \Sigma^\prime, v \nmid p$ and $H^1(K_{\infty,v},T^\dagger)/H^1_{I_v}(K_{\infty,v},T^\dagger)$ when $v | p$. Let also $J_v^n$ denote $H^1(K_{n,v},T^\dagger)/H^1_f(K_{n,v},T^\dagger)$ when $v\in \Sigma_n,v \nmid p$ and $H^1(K_{n,v},T^\dagger)/H^1_{I_v}(K_{n,p},T^\dagger)$ when $v | p$. Consider the commutative diagram with exact rows
\begin{equation*}
\begin{tikzcd}
0 \arrow{r} & \Sel_{\underline{I}}^\text{str}(T^\dagger /K_{n}) \arrow{r} \arrow[d, "\text{res}"] & H^1_{\Sigma_n}(K_n,T^\dagger) \arrow{r} \arrow[d, "\text{res}"] & \prod_v J_v^n \arrow{r} \arrow[d, "\text{res}"] & 0 \\
0 \arrow{r} & \Sel_{\underline{I}}(T^\dagger /K_\infty)^{\Gamma_n} \arrow{r} & H^1_{\Sigma^\prime}(K_\infty, T^\dagger)^{\Gamma_n} \arrow{r} & \left( \prod_v J_v \right)^{\Gamma_n} \arrow{r} & 0
\end{tikzcd}
\end{equation*}
By the snake lemma, the kernel of the restriction map $\Sel_{\underline{I}}^\text{str}(T^\dagger /K_{n}) \to \Sel_{\underline{I}}(T^\dagger /K_\infty)^{\Gamma_n}$ is contained in the kernel of the restriction map  $H^1_{\Sigma_n}(K_n,T^\dagger) \to H^1_{\Sigma^\prime}(K_\infty, T^\dagger)^{\Gamma_n}$. By inflation-restriction, this kernel is isomorphic to $H^1(\Gamma_n, H^0(K_\infty, T^\dagger))$. Since the Galois group of $K_{\infty,\gq}$ over $K_\gq$ is a pro-$p$ group, the orbit-stabilizer theorem together with \textbf{(Tors)} imply that $H^0(K_{\infty,\gq},T^\dagger)$ is trivial. We get $H^0(K_\infty, T^\dagger)=0$ and thus the restriction map of interest is an injection.
\end{proof}

\begin{lemma}\label{Y}
Suppose that $X_{\underline{I}}(T^\dagger /K_\infty)$ and $X_{\underline{I}^c}(T^\ast (1)^\dagger /K_\infty)$ are torsion over $\Lambda(\Gamma)$. Then we have $[Y_{\underline{I}}(T^\dagger/K_\infty)] = [Y_{\underline{I}}(T^\dagger/K_\infty)]_\text{si}$ and $[Y_{\underline{I}^c}(T^\ast (1)^\dagger/K_\infty)] = [Y_{\underline{I}^c}(T^\ast (1)^\dagger/K_\infty)]_\text{si}$.
\end{lemma}

\begin{proof}
In \cite[Theorem 4.1.3]{LLTT}, it is shown that there exists relatively prime simple elements $f_1,\ldots, f_m \in \Lambda(\Gamma)$ such that
$$
f_1 \cdots f_m \left( \Sel_{\underline{I}}(T^\dagger/K_{\infty})^{\Gamma_n} \right)_\text{div} = 0
$$
for every $n$. By lemma \ref{restriction}, we can see $\Sel_{\underline{I}}^\text{str}(T^\dagger /K_{n})$ as a $\Lambda(\Gamma)$-submodule of the module $\Sel_{\underline{I}}(T^\dagger /K_\infty)^{\Gamma_n}$. Thus, the product $f_1 \cdots f_m$ also annihilates the module $\Sel_{\underline{I}}^\text{str}(T^\dagger /K_{n})_\text{div}$ for every $n$. By the definition of the action of $\Lambda(\Gamma)$ on inverse limit, we see that 
$$
f_1\cdots f_m Y_{\underline{I}}(T^\dagger/K_\infty) =0.
$$
\end{proof}

\begin{theorem}\label{non-simple parts}
Suppose that \textbf{(H.crys)}, \textbf{(H.HT)}, \textbf{(H.Frob)} \textbf{(H.P)}, \textbf{(H.Ram)}, \textbf{(H.F)} and \textbf{(Tors)} hold. Suppose that $X_{\underline{I}}(T^\dagger /K_\infty)$ and $X_{\underline{I}^c}(T^\ast (1)^\dagger /K_\infty)$ are torsion over $\Lambda(\Gamma)$. We have 
$
[X_{\underline{I}^c}(T^\ast (1)^\dagger /K_\infty)]_\text{ns}^\iota = [X_{\underline{I}}(T^\dagger /K_\infty)]_\text{ns}
$.
\end{theorem}

\begin{proof}
We follow the proof of \cite[Theorem 3.3]{AL21}. Consider the commutative diagram
\begin{equation*}
\begin{tikzcd}
0 \arrow{r} & \Sel_{\underline{I}}^\text{str}(T^\dagger / K_n)_\text{div} \arrow{r} \arrow{d} & \Sel_{\underline{I}}^\text{str}(T^\dagger / K_n) \arrow{d} \arrow{r} & \mathfrak{a}_n \arrow{r} \arrow{d} & 0 \\
0 \arrow{r} & \Sel_{\underline{I}}^\text{str}(T^\dagger / K_\infty)_\text{div}^{\Gamma_n} \arrow{r} & \Sel_{\underline{I}}(T^\dagger / K_\infty)^{\Gamma_n} \arrow{r} & \left( \varinjlim_n \mathfrak{a}_n \right)^{\Gamma_n} 
\end{tikzcd}
\end{equation*}
with exact rows. By definition, the kernel of the rightmost map is $\mathfrak{a}_n^0$. The snake lemma shows that $\mathfrak{a}_n^0$ can be seen as a submodule of $\Sel_{\underline{I}}^\text{str}(T^\dagger / K_\infty)_\text{div}^{\Gamma_n}/\Sel_{\underline{I}}^\text{str}(T^\dagger / K_n)_\text{div}$ because the middle map is an injection by lemma \ref{restriction}. Since $\Sel_{\underline{I}}(T^\dagger / K_\infty)^{\Gamma_n}$ is annihilated by a product of simple elements, $\mathfrak{a}_n^0$ is also annihilated by that product for every $n$. Thus, $[\mathfrak{a}^0] = [\mathfrak{a}^0]_\text{si}$. By lemma \ref{prime}, $[\mathfrak{a}]_\text{ns} = [\mathfrak{a}^\prime]_\text{ns}$. In a similar fashion, one also find that $[\mathfrak{b}]_\text{ns} = [\mathfrak{b}^\prime]_\text{ns}$. Combining the exact sequence \eqref{main es} with lemma \ref{Y}, one gets $[\mathfrak{b}]_\text{ns} = [X_{\underline{I}}(T^\dagger /K_\infty)]_\text{ns}$ and $[\mathfrak{a}]_\text{ns} = [X_{\underline{I}^c}((T^\ast (1))^\dagger /K_\infty)]_\text{ns}$. By proposition \ref{non simple},
$$
[X_{\underline{I}^c}(T^\ast (1)^\dagger /K_\infty)]_\text{ns}^\iota = [X_{\underline{I}}(T^\dagger /K_\infty)]_\text{ns}.
$$
\end{proof}

\subsection{Abelian varieties}\label{sct:Abelian}
To get an analogue of theorem \ref{non-simple parts} for simple parts, we need to specialize to the case where $T$ is the $p$-adic Tate module of an abelian variety. The case of potentially ordinary reduction at primes above $p$ was carried out in \cite[Theorem 4.4.4]{LLTT}. Here, we treat the case of supersingular reduction. \\
Let $A$ be an abelian variety defined over $K$ of dimension $g$. Suppose that $A$ has good supersingular reduction at both prime above $p$. We make the assumption: \bigbreak
\noindent \textbf{(P)} $A$ admits a polarization. \bigbreak
We will denote by $\alpha: A \to A^t$ this polarization where $A^t$ is the dual abelian variety. Let $T=T_p(A) \colonequals \varprojlim_n A[p^n]$ be the $p$-adic Tate module of $A$ which is a free $\Zp$-module of rank $d=2g$. Then, $T^\ast (1) = T_p(A^t)$, $T^\dagger = A^t[p^\infty]$ and $T^\ast (1)^\dagger = A[p^\infty]$. Furthermore, $T \otimes \Qp$ satisfies \textbf{(H.HT)}, \textbf{(H.Frob)} and \textbf{(H.P)}. The hypothesis \textbf{(Tors)} is satisfied by supersingularity. The fact that $T$ satisfies \textbf{(H.F)} for any subextension of $F^\mathrm{cyc}$ is the main result of \cite{Ima75}. The isogeny $\alpha$ extends to a map between Tate modules $T_p(A)\to T_p(A^t)$ which we also denote $\alpha$. Since this isogeny is defined over $K$, it is $G_K$-equivariant and is in fact a morphism of $\mathbf{Z}_p[G_{K_\mathfrak{q}}]$-modules. We further extend the map $\alpha$ to $\mathcal{M}_p = T \otimes_{\mathbf{Z}_p} \mathbf{Q}_p$ in the natural way. We get by functoriality a morphism of filtered $\varphi$-modules $\widetilde{\alpha}: \mathbb{D}_\text{cris}(T) \to \mathbb{D}_\text{cris}(T^\ast (1))$. As in section \ref{Coleman}, we choose a $\mathbf{Z}_p$-basis $\{v_{\gq,1},\ldots,v_{\gq,2g}\}$ of $\mathbb{D}_{\cris,\gq}(T)$ such that $\{v_{\gq,1},\ldots,v_{\gq,g}\}$ is a $\mathbf{Z}_p$-basis of $\Fil^0\mathbb{D}_{\cris,\gq}(T)$. Let $\{v_{\gq,1}^\prime,\ldots,v_{\gq,2g}^\prime\}$ be the dual basis of $\mathbb{D}_{\cris,\gq}(T^\ast (1))$ with respect to the pairing $[,]$. Let 
$$
(\cdot,\cdot):\mathbb{D}_{\cris,\gq}(T)\times \mathbb{D}_{\cris,\gq}(T) \to \mathbf{Z}_p
$$
be the pairing defined by $(x,y) \colonequals [x,\widetilde{\alpha}(y)]$. The pairing $(\cdot,\cdot)$ is a nondegenerate skew-symmetric bilinear form since $\alpha$ is a polarization. Thus there exists a symplectic basis $\{X_{\gq,1},\ldots,X_{\gq,g},Y_{\gq,1},\ldots,Y_{\gq,g}\}$ of $\mathbb{D}_{\cris,\gq}(T)$. If we write $\delta_{ij}$ for the Kronecker delta, then this basis satisfies $(X_{\gq,i},Y_{\gq,j}) = \delta_{ij}$ and $(X_{\gq,i},X_{\gq,j})=0=(Y_{\gq,i},Y_{\gq,j})$.

\begin{lemma}\label{SHT basis}
The module $\mathbb{D}_{\cris,\gq}(T)$ admits a symplectic basis that is Hodge-compatible.
\end{lemma}

\begin{proof}
Start with a $\mathbf{Z}_p$-basis $\{X_{\gq,1},\ldots,X_{\gq,g}\}$ of $\Fil^0\mathbb{D}_{\cris,\gq}(T)$. Since $\widetilde{\alpha}$ respects filtrations, 
$$
\widetilde{\alpha}(\Fil^0\mathbb{D}_{\cris,\gq}(T)) \subseteq \Fil^0\mathbb{D}_{\cris,\gq}(T^\ast (1)).
$$
But those submodules are orthogonal complement of each other with respect to the pairing $[\cdot,\cdot]$. Thus, $(X_{\gq,i},X_{\gq,j})=0$, i.e. $\Fil^0 \mathbb{D}_{\cris,\gq}(T)$ is isotropic. Moreover, the fact that $\mathcal{M}_p$ satisfies \textbf{(H.P)} gives that $\text{rk}\left( \Fil^0 \mathbb{D}_{\cris,\gq}(T) \right) = \text{rk}\left( \mathbb{D}_{\cris,\gq}(T) \right)/2$. We can now apply \cite[Corollary 3.13]{HKK14} to deduce that $\Fil^0 \mathbb{D}_{\cris,\gq}(T)$ is a Lagrangian submodule and follow with the proof of \cite[Theorem 3.14]{HKK14} to show that $\{X_{\gq,1},\ldots,X_{\gq,g}\}$ can be extended to a symplectic basis $\{X_{\gq,1},\ldots,X_{\gq,g},Y_{\gq,1},\ldots,Y_{\gq,g}\}$ of $\mathbb{D}_{\cris,\gq}(T)$.
\end{proof}

\begin{remark}
Note that in \cite{HKK14}, the authors consider modules over the ring of Colombeau-generalized numbers. However, the proof of the results of section 3 that are of interest to us works as well for $\mathbf{Z}_p$-modules. 
\end{remark}

\begin{lemma}
Let $\mathcal{B}_\gq=\{X_{\gq,1},\ldots,X_{\gq,g},Y_{\gq,1},\ldots,Y_{\gq,g}\}$ be a Hodge-compatible symplectic basis of $\mathbb{D}_{\cris,\gq}(T)$ whose existence is guaranteed by lemma \ref{SHT basis}. Let $\mathcal{B}^\ast_\gq$ be the dual basis of $\mathbb{D}_{\cris,\gq}(T^\ast (1))$ with respect to $[\cdot,\cdot]$. Then,  
$$
\left[
\begin{array}{c|c}
0 & I_{g} \\
\hline
-I_{g} & 0
\end{array}
\right]
\widetilde{\alpha}(\mathcal{B}_\gq) = \mathcal{B}^\ast_\gq.
$$
\end{lemma}

\begin{proof}
One can check that the dual basis is given by
$$
\mathcal{B}^\ast_\gq = \{\widetilde{\alpha}(Y_{\gq,1}),\ldots, \widetilde{\alpha}(Y_{\gq,g}), -\widetilde{\alpha}(X_{\gq,1}),\ldots,-\widetilde{\alpha}(X_{\gq,g})\}.
$$
Then we see that 
$$
\widetilde{\alpha}(\mathcal{B}_\gq) = \{\widetilde{\alpha}(X_{\gq,1}),\ldots, \widetilde{\alpha}(X_{\gq,g}),\widetilde{\alpha}(Y_{\gq,1}),\ldots, \widetilde{\alpha}(Y_{\gq,g}) \}
$$
differs from $\mathcal{B}^\ast_\gq$ by the matrix in the statement.
\end{proof}

%

The map $\alpha$ induces another map $\alpha^\ast : H^1(K_{n,\gq},T) \to H^1(K_{n,\gq},T^\ast (1))$ which is compatible with corestriction allowing us to consider $\alpha^\ast : H^1_\text{Iw}(K_{\infty,\gq},T) \to H^1_\mathrm{Iw}(K_{\infty,\gq},T^\ast (1))$. Extend $\widetilde{\alpha}$ $\mathcal{H}_{\widehat{F}_\infty}(\Gamma_p)$-linearly to
$$
\mathcal{H}_{\widehat{F}_\infty}(\Gamma_p) \otimes_{\mathbf{Z}_p} \mathbb{D}_{\cris,\gq}(T) \to  \mathcal{H}_{\widehat{F}_\infty}(\Gamma_p) \otimes_{\mathbf{Z}_p} \mathbb{D}_{\cris,\gq}(T^\ast (1)).
$$

\begin{proposition}\label{Fontaine}
The following diagram is commutative:
$$
\begin{tikzcd}
H^1_\mathrm{Iw}(K_{\infty,\gq},T) \arrow[r,"\mathcal{L}_{T,\gq}"] \arrow[d, "\alpha^\ast"]
& \mathcal{H}_{\widehat{F}_\infty}(\Gamma_p) \otimes_{\mathbf{Z}_p} \mathbb{D}_{\cris,\gq}(T) \arrow[d, "\widetilde{\alpha}"] \\
H^1_\mathrm{Iw}(K_{\infty,\gq},T^\ast (1)) \arrow[r,  "\mathcal{L}_{T^\ast (1),\gq}"]
& \mathcal{H}_{\widehat{F}_\infty}(\Gamma_p) \otimes_{\mathbf{Z}_p} \mathbb{D}_{\cris,\gq}(T^\ast (1))
\end{tikzcd}
$$
\end{proposition}

\begin{proof}
We need to show that the maps induced by $\alpha$ commute with $h^1_{\mathrm{Iw},T}$, $1-\varphi$ and $\mathfrak{M}\otimes 1$. First, we show that the maps induced by the map $\alpha$ commute with Fontaine's isomorphism of $\Lambda(\Gamma_p)$-modules
$$
h^1_{\infty,T}: \mathbb{N}_{F_\infty}(T)^{\psi=1} \xrightarrow{\sim} H^1_\mathrm{Iw}(F_\infty(\mu_{p^\infty}),T)
$$
where $\mathbb{N}_{F_\infty}(T) \colonequals \mathbb{N}(T) \widehat{\otimes}_{\mathcal{O}_F} S_{F_\infty /F}$. If $x\in 1+p\mathbf{Z}_p$, we choose the largest integer $k\geq 1$ such that $x \in p^k\mathbf{Z}_p^\times$. We define $\log_p^0(x) \colonequals \frac{\log_p(x)}{p^k}$. If $y \in \mathbb{D}(T)^{\psi=1}$, then there exists $b \in \mathbf{A}\otimes_{\mathbf{Z}_p}T$ such that $(\gamma_0^\text{cyc}-1)(\varphi-1)b=(\varphi-1)y$. The isomorphism $h^1_{\mathrm{Iw},T}$ is build up from the maps
$$
h^1_{n,T}: \mathbb{D}(T)^{\psi=1} \to H^1(\Qp(\mu_{p^n}),T)
$$
defined by $h^1_{n,T}(y)(\sigma) \colonequals \log^0_p(\chi(\gamma_0^{\text{cyc}})) \left[\sigma \mapsto \frac{\sigma-1}{\gamma-1}y - (\sigma-1)b \right]$ for any $y\in \mathbb{D}(T)^{\psi=1}$ and $\sigma \in H^1(\Qp(\mu_{p^n}),T)$. The existence of $b$ as well as the fact that $h^1_{n,T}$ defines a cocycle is shown in \cite[Proposition I.8]{Be03}. The $h^1_{n,T}$ are compatible with corestriction $H^1(\Qp(\mu_{p^n}),T)\to H^1(\Qp(\mu_{p^{n-1}}),T)$ and give an isomorphism \cite[Theorem II.8]{Be03}
\begin{align}\label{Fontaine isomorphism}
h^1_{\mathrm{Iw},T}:\mathbb{D}(T)^{\psi=1} & \to H^1_\mathrm{Iw}(\Qp(\mu_{p^\infty}),T) \\
y &\mapsto \varprojlim_n h^1_{n,T}(y). \nonumber
\end{align}
If $\widetilde{\alpha}$ is the map $\mathbb{D}(T)^{\psi=1}\to \mathbb{D}(T^\ast (1))^{\psi=1}$ given by functoriality, then $\widetilde{\alpha}$ commutes with $\varphi$ and the action of $\Gamma_0^\text{cyc}$. Therefore, if $b \in \mathbf{A}\otimes_{\mathbf{Z}_p}T$ is a solution of $(\gamma_0^\text{cyc}-1)(\varphi-1)b=(\varphi-1)y$, then $\widetilde{\alpha}(b) \in \mathbf{A}\otimes_{\mathbf{Z}_p}T^\ast (1)$ will be a solution of $(\gamma_0^\text{cyc}-1)(\varphi-1)\widetilde{\alpha}(b)=(\varphi-1)\widetilde{\alpha}(y)$. It follows that
\begin{align*}
h^1_{n,T^\ast (1)}(\widetilde{\alpha}(y))(\sigma) &= \log^0_p(\chi(\gamma_0^{\text{cyc}}))\left[\sigma \mapsto \frac{\sigma-1}{\gamma-1}\widetilde{\alpha}(y) - (\sigma-1)\widetilde{\alpha}(b) \right] \\
& = \alpha^\ast \circ h^1_{n,T}.
\end{align*}
Because $A$ is supersingular at both primes above $p$, $T$ is irreducible and thus has no quotient isomorphic to $\Qp$. So $h^1_\text{Iw,T}$ is really a map from $\mathbb{N}(T)^{\psi=1}$ to $H^1_\text{Iw}(\Qp(\mu_{p^\infty}),T)$. By passing to the limit, we conclude that the maps induced by $\alpha$ commutes with the isomorphism \eqref{Fontaine isomorphism}. The map $h^1_{\infty,T}$ of Loeffler--Zerbes is then constructed via the inverse limit of the $h^1_{\text{Iw},T}$ in the unramified tower $F_\infty(\mu_{p^\infty}) / F(\mu_{p^\infty})$ (see \cite[Proposition 4.5]{LZ14}) and thus also commutes with the maps induced by $\alpha$. Moreover, it is clear that $\widetilde{\alpha}$ commutes with 
$$
\mathbb{N}_{F_\infty}(T)^{\psi=1}\xrightarrow{1-\varphi} (\varphi^\ast \mathbb{N}(T))^{\psi=0} \widehat{\otimes}_{\mathbf{Z}_p} S_{F_\infty/F} \hookrightarrow S_{F_\infty/F}\widehat{\otimes}_{\mathbf{Z}_p} \Lambda(\Gamma^\text{cyc})\otimes_{\Zp}\mathbb{D}_{\cris,\gq}(T).
$$

\end{proof}

\begin{proposition}\label{coleman functorial}
Let $\mathcal{B}_\gq$ be a Hodge-compatible symplectic basis of $\mathbb{D}_{\cris,\gq}(T)$. Suppose that the matrix $C_\gq$ with respect to this basis is block anti-diagonal. Let $z \in H^1_\mathrm{Iw}(K_{\infty,\gq},T)$. Then,
$$
D \cdot \text{Col}_T^\gq(z) = \left[
\begin{array}{c|c}
0 & I_{g} \\
\hline
-I_{g} & 0
\end{array}
\right] \cdot \text{Col}_{T^\ast (1)}^\gq(\alpha^\ast(z))
$$
where $D \in \mathrm{GL}_{2g}\left(\mathrm{Frac}(\mathcal{H}(\Gamma^\text{cyc}))\right)$ is a block diagonal matrix.
\end{proposition}

\begin{proof}
To simplify the notation, write $B$ for the change of basis matrix $\left[
\begin{array}{c|c}
0 & I_{g} \\
\hline
-I_{g} & 0
\end{array}
\right]$. The hypothesis on $C_\gq$ implies that $M_{T,\gq}$ is also block anti-diagonal since $M_{T,\gq}$ is the limit of $2n+1$ block anti-diagonal matrices. The same is true for $M_{T^\ast (1),\gq}$. By \eqref{regulator decomposition} and proposition \ref{Fontaine}, we have
\begin{align*}
\widetilde{\alpha} ( (X_{\gq,1},\ldots,X_{\gq,g}&,Y_{\gq,1},\ldots, Y_{\gq,g})\cdot M_{T,\gq} \cdot \text{Col}_T^\gq(z) )\\
&= (X_{\gq,1}^\prime,\ldots,X_{\gq,g}^\prime,Y_{\gq,1}^\prime,\ldots, Y_{\gq,g}^\prime)\cdot M_{T^\ast(1),\gq}\cdot \text{Col}_{T^\ast (1)}^\gq(\alpha^\ast (z)).
\end{align*}
By change of basis on the right hand side and since $\widetilde{\alpha}$ is $\mathcal{H}_{\widehat{F}_\infty}(\mathcal{G})$-linear, we get
\begin{align*}
(\widetilde{\alpha}(&X_{\gq,1}),\ldots,\widetilde{\alpha}(X_{\gq,g}),\widetilde{\alpha}(Y_{\gq,1}),\ldots, \widetilde{\alpha}(Y_{\gq,g})) \cdot M_{T,\gq}\cdot \text{Col}_T^\gq(z) \\&= (\widetilde{\alpha}(X_{\gq,1}),\ldots,\widetilde{\alpha}(X_{\gq,g}),\widetilde{\alpha}(Y_{\gq,1}),\ldots, \widetilde{\alpha}(Y_{\gq,g})) B M_{T^\ast (1),\gq}B^{-1}\cdot B \cdot \text{Col}_{T^\ast(1)}^\gq(\alpha^\ast (z)).
\end{align*}
From now on, we omit the reference to the basis $\widetilde{\alpha}(\mathcal{B}_\gq)$. The matrix $M_{T^\ast (1),\gq}$ is not invertible over $\Lambda(\Gamma^\mathrm{cyc}$), but by proposition \ref{determinant} (with $T$ replaced by $T^\ast (1)$), we may invert it in $\mathrm{GL}_{2g}(\mathrm{Frac}(\mathcal{H}(\Gamma^\mathrm{cyc})))$. We conclude that
$$
D \cdot \text{Col}_T^\gq(z) = B \cdot \text{Col}_{T^\ast (1)}^\gq(\alpha^\ast (z))
$$
where $D = BM_{T^\ast (1),\gq}^{-1}B^{-1}M_{T,\gq}$ is a block diagonal matrix.
\end{proof}

\begin{corollary}\label{kernels}
Let $\mathcal{B}_\gq$ be a Hodge-compatible symplectic basis of $\mathbb{D}_{\cris,\gq}(T)$. Suppose that the matrix $C_\gq$ with respect to this basis is block anti-diagonal. Let $I_\gq$ be either the set $\{1,2,\ldots,g\}$ or $\{g+1,g+2,\ldots,2g\}$. Then, $z \in \ker \text{Col}_{T,I_\gq}^\gq$ if and only if $\alpha^\ast(z) \in \ker \text{Col}_{T^\ast (1),I_\gq^c}^\gq$.
\end{corollary}

\begin{proof}
By proposition \ref{coleman functorial}, we can write
$$
\left[
\begin{array}{c|c}
D_1 & 0 \\
\hline
0 & D_2
\end{array}
\right] \cdot \text{Col}_T^\gq(z) = B\cdot \text{Col}_{T^\ast (1)}^\gq(\alpha^\ast(z))
$$
for some $D_1,D_2 \in \mathrm{GL}_{g}\left(\mathrm{Frac}(\mathcal{H}(\Gamma^\mathrm{cyc}))\right)$. Thus, the vector $(\text{Col}_{T,1}^\gq(z), \ldots, \text{Col}_{T,g}^\gq(z))^t$ is the zero vector if and only if $(\text{Col}_{T^\ast(1),g+1}^\gq(\alpha^\ast(z)),\ldots, \text{Col}_{T^\ast(1),2g}^\gq(\alpha^\ast(z)))^t$ is the zero vector. Similarly, $(\text{Col}_{T,g+1}^\gq(z), \ldots, \text{Col}_{T,2g}^\gq(z))^t=0$ if and only if 
$$
(-\text{Col}_{T^\ast(1),1}^\gq(\alpha^\ast(z)),\ldots, -\text{Col}_{T^\ast(1),g}^\gq(\alpha^\ast(z)))^t=0.
$$
\end{proof}

\begin{theorem}\label{Selmer functoriality}
Suppose that there exists a polarization $\alpha:A \to A^t$. For $\gq \in \{\gp,\gp^c\}$, let $\mathcal{B}_\gq$ be a Hodge-compatible symplectic basis of $\mathbb{D}_{\cris,\gq}(T)$. Suppose that the matrix $C_\gq$ with respect to this basis is block anti-diagonal. Let $I_\gq$ be either the set $\{1,2,\ldots,g\}$ or $\{g+1,g+2,\ldots,2g\}$. Then, $\alpha^\ast \left( \Sel_{\underline{I}}(T^\dagger/K_\infty) \right) \subseteq \text{Sel}_{\underline{I}^c}(T^\ast(1)^\dagger/K_\infty)$.
\end{theorem}

\begin{proof}
Let $z \in \Sel_{\underline{I}}(T^\dagger/K_\infty)$. By definition, this means that $z \in H^1_{\Sigma^\prime}(K_\infty,T^\dagger)$ is such that its localization at the primes above $p$ in $K_\infty$ lands in $H^1_{I_\gq}(K_{\infty,\gq},T^\dagger)$ and its localization at the primes $v$ not dividing $p$ lands in $H^1_f(K_{\infty,v},T^\dagger)$. We separate the proof in two steps depending on whether or not $v$ divides $p$. \\
\textbf{Case 1 : $v$ divides $p$.} By hypothesis,
$$
\mathrm{res}_\gq(z) \in H^1_{I_\gq}(K_{\infty,\gq},T^\dagger) = \left( \ker \col_{T,I_\gq}^\gq \right)^\perp.
$$
It follows that $\langle \mathrm{res}_\gq(z), x \rangle_{\text{Tate}} = 0$ for all $x \in \ker \col_{T,I_\gq}^\gq$ where $\langle \sim,\sim \rangle_\text{Tate}$ is the local Tate pairing
$$
H^1_\mathrm{Iw}(K_{\infty,\gq},T) \times H^1(K_{\infty,\gq}, T^\dagger) \to \mathbf{Q}_p / \mathbf{Z}_p.
$$
Functoriality of cup products implies that $$\langle \alpha^\ast (\mathrm{res}_\gq(z)),\alpha^\ast(x)\rangle_{Tate}=0$$ for all $x\in \ker \col_{T,I_\gq}^\gq$. Corollary \ref{kernels} tells us that $\alpha^\ast(x) \in \ker \col_{T^\ast(1),I_\gq^c}^\gq$. By surjectivity of $\alpha^\ast$ (since $\alpha$ is an isogeny), we know that all $y \in \ker \col_{T^\ast(1),I_\gq^c}^\gq$ is of the form $\alpha^\ast(x^\prime)$ for some $x^\prime \in H^1_\mathrm{Iw}(K_{\infty,\gq},T)$. Again by corollary \ref{kernels}, we get that $x^\prime$ is in fact in $\ker \col_{T,I_\gq}^\gq$. Hence, $\langle \alpha^\ast(\mathrm{res}_\gq(z)), y \rangle_\text{Tate}=0$ for all $y \in \ker \col_{T^\ast(1),I_\gq^c}^\gq$. In this case, we conclude that $\mathrm{res}_\gq(\alpha^\ast(z))=\alpha^\ast(\mathrm{res}_\gq(z)) \in H^1_{I_\gq^c}(K_{\infty,\gq}, T^\ast(1)^\dagger)$ like we wanted. \\
\textbf{Case 2 : $v$ does not divide $p$.} Let $v$ be a prime in $K_\infty$ such that $v \nmid p$. The local condition $H^1_f(K_{\infty,v},T^\dagger)$ is defined as the direct limit of the unramified subgroups
$$
H^1_f(K_{n,w},T^\dagger)=\ker \left( H^1(K_{n,w},T^\dagger)\xrightarrow{\mathrm{res}_w}H^1(I_{n,w},(T^\dagger)^{\Gal(K_{n,w}^\mathrm{ur}/K_{n,w})}\right)
$$
where $w$ is a prime under $v$ in $K_{n}$ and $I_{n,w}$ is the inertia subgroup of $w$ inside the Galois group $\Gal(\overline{K_{n,w}}/K_{n,w})$. Since $\alpha^\ast$ commutes with $\mathrm{res}_w$, we get $\mathrm{res}_v(\alpha^\ast(z)) \in H^1_f(K_{\infty,v},T^\ast (1)^\dagger)$. 
Combining cases 1 and 2, it follows that $$\mathrm{res}(\alpha^\ast(z)) \in \mathcal{P}_{\Sigma,\underline{I}^c}(T^\ast (1)^\dagger/K_\infty).$$
\end{proof}

The map $\alpha^\ast$ of theorem \ref{Selmer functoriality} induces a $\Lambda(\Gamma)$-modules homomorphism $$\alpha^\vee: X_{\underline{I}^c}(T^\ast(1)^\dagger/K_\infty) \to X_{\underline{I}}(T^\dagger/K_\infty).$$
As in \cite[Lemma 4.3.1]{LLTT}, we get the following corollary:

\begin{corollary}\label{cor:both torsion}
The $\Lambda(\Gamma)$-module $X_{\underline{I}}(T^\dagger/K_\infty)$ is torsion if and only if the $\Lambda(\Gamma)$-module $X_{\underline{I}^c}(T^\ast(1)^\dagger/K_\infty)$ is torsion.
\end{corollary}

\begin{proof}
The homomorphism $\alpha^\vee$ has kernel and cokernel annihilated by $\deg(\alpha)$.
\end{proof}

\subsection{Simple parts and proof of main result}

\begin{corollary}\label{simple parts}
Keep the same hypotheses as theorem \ref{Selmer functoriality}. If $X_{\underline{I}}(T^\dagger/K_\infty)$ is $\Lambda$-torsion, then $[X_{\underline{I}}(T^\dagger/K_\infty)]_\text{si}=[X_{\underline{I}^c}(T^\ast(1)^\dagger/K_\infty)]_\text{si}$.
\end{corollary}

\begin{proof}
We translate the proof of \cite[Corollary 4.3.2]{LLTT} in our setting. By corollary \ref{cor:both torsion}, $X_{\underline{I}^c}(T^\ast(1)^\dagger/K_\infty)$ is also torsion. To simplify notation, write $X_{\underline{I}}({-})$ for $X_{\underline{I}}({-}/K_\infty)$. Define $\phi : [X_{\underline{I}^c}(T^\ast(1)^\dagger)]_\text{si} \to [X_{\underline{I}}(T^\dagger)]_\text{si}$ by the composition
$$
[X_{\underline{I}^c}(T^\ast(1)^\dagger)]_\text{si} \hookrightarrow X_{\underline{I}^c}(T^\ast(1)^\dagger) \xrightarrow{\phi_1} X_{\underline{I}}(T^\dagger) \xrightarrow{\phi_2} [X_{\underline{I}}(T^\dagger)] \xrightarrow{\phi_3} [X_{\underline{I}}(T^\dagger)]_\text{si}
$$
where $\phi_1$ is the map $\alpha^\vee$ whose existence follows from theorem \ref{Selmer functoriality}, $\phi_2$ is a pseudo-isomorphism and $\phi_3$ is a surjection. We must show that $\phi$ is a pseudo-injection. Suppose that $[X_{\underline{I}^c}(T^\ast(1)^\dagger)]_\text{si}$ is annihilated by $f \in \Lambda(\Gamma)$ which is a product of simple elements. The kernel of each $\phi_i$ is annihilated by some $g_i \in \Lambda$ relatively prime to $f$. Thus the kernel of $\phi$ is annihilated by both $f$ and $g_1g_2g_3$. We then apply \cite[Lemma 2.1.1]{LLTT} which tells us that a finitely generated $\Lambda(\Gamma)$-module $M$ is pseudo-null if and only if there exists relatively prime $f_1,\ldots,f_k \in \Lambda(\Gamma)$, $k \geq 2$ such that $f_i M=0$ for every $i$. Hence, $\phi$ is a pseudo-injection. \\
The relation $C_{\varphi,\gq}^\ast = \frac{1}{p}(C_{\varphi,\gq}^{-1})^t$ implies that $C_\gq^\ast$ is block anti-diagonal if $C_\gq$ is. Apply theorem \ref{Selmer functoriality}  with $\alpha$ replaced by the dual isogeny $\alpha^t:A^t\to A$, $C_\gq$ replaced by $C_\gq^\ast$ and $\underline{I}$ replaced by $\underline{I}^c$. It follows that we also have a pseudo-injection $[X_{\underline{I}}(T^\dagger)]_\text{si} \to [X_{\underline{I}^c}(T^\ast(1)^\dagger)]_\text{si}$. By \cite[Lemma 2.1.3]{LLTT}, $\phi$ is a pseudo-isomorphism.
\end{proof}

\begin{corollary}
Suppose that $X_{\underline{I}}(T^\dagger/K_\infty)$ is $\Lambda$-torsion. Suppose that there exists a polarization $\alpha:A \to A^t$. For $\gq \in \{\gp,\gp^c\}$, let $\mathcal{B}_\gq$ be a Hodge-compatible basis of $\mathbb{D}_{\cris,\gq}(T)$. Suppose that both matrices $C_\gp$ and $C_{\gp^c}$ are block anti-diagonal. Let $I_\gq$ be either the set $\{1,2,\ldots,g\}$ or $\{g+1,g+2,\ldots,2g\}$. Then,
$$
\Sel_{\underline{I}}(A[p^\infty]/K_\infty)^\vee \sim \Sel_{\underline{I}^c}(A^t[p^\infty]/K_\infty)^{\vee,\iota}.
$$
\end{corollary}

\begin{proof}
Suppose first that $\mathcal{B}_{\gq}$ is a Hodge-compatible symplectic basis whose existence is guaranteed by lemma \ref{SHT basis}. Then, the result follow by putting together theorem \ref{non-simple parts} and corollary \ref{simple parts}. Furthermore, the hypotheses imply that $\Sel_{\underline{I}}(A[p^\infty]/K_\infty)$ and $\Sel_{\underline{I}^c}(A^t[p^\infty]/K_\infty)$ do not depend on the choice of Hodge-compatible bases \cite[Proposition 4.17]{DR21}. Since the functional equation holds for a Hodge-compatible symplectic basis, it holds for all Hodge-compatible bases.
\end{proof}

\begin{remark}
The discussion of \cite[Section 3.3]{LP20} shows that the hypothesis on $C_\gq$ is satisfied for a certain class of abelian varieties of $\mathrm{GL}_2$-type.
\end{remark}

\section{Chromatic Selmer groups}\label{chromatic}

In this last section, we investigate the case when $T=T_p(E)$ is the $p$-adic Tate module of a supersingular elliptic curve. In this setting, Selmer groups were constructed by Kobayashi \cite{Kob03} (when $a_p=0$) and generalized by Sprung \cite{Spr12} (when $p|a_p$). The Selmer groups constructed by Sprung are the so-called \textit{chromatic Selmer groups}. The strategy used to prove the functional equation for chromatic Selmer groups over $\Zp^2$-extensions is to compare chromatic Coleman maps to multi-signed Coleman maps. This will enable us to use the reciprocity formula of Loeffler--Zerbes in order to deduce the required orthogonality conditions. 

\subsection{Chromatic Selmer Selmer groups over \texorpdfstring{$\Zp^2$-extensions}{Zp2-extensions}}

We first review the construction of two-variable chromatic Selmer groups as in \cite{Spr16}. Recall that we write $F_m$ for the unique subextension of degree $p^m$ contained in the unramified $\Zp$-extension of $\Qp$. We also write $F_m^n$ for the unique subextension of degree $p^n$ inside $F_m(\mu_{p^n})$. Denote by $\mathfrak{m}_{n,m}$ the maximal ideal in $\mathcal{O}_{F_m^n}$. Let $\Lambda_{n,m} \colonequals \Zp[\Gal (F_m^n/\Qp)]$.

\begin{definition}
For $x \in H^1(F_m^n,T)$, we define a paring $P_{(n,m),x}:H^1(F_m^n,T)\to \Lambda_{n,m}$ by
$$
z \mapsto \sum_{\sigma \in \Gal(F_m^n/\Qp)}(x^\sigma,z)_{n,m}\cdot \sigma
$$
where $(\sim,\sim )$ is the pairing coming from the cup product induced by the Weil pairing.
\end{definition}

Let $\mathcal{A}_i \colonequals \begin{bmatrix} a_p & \Phi_i(1+X) \\ -1 & 0 \end{bmatrix}$.
\begin{definition}
We put $\mathcal{H}_n \colonequals \mathcal{A}_1 \cdots \mathcal{A}_n$ and define the endomorphism $h_{n,m}:\Lambda_{n,m}^{\oplus 2}\to \Lambda_{n,m}^{\oplus 2}$ by
$$
(a,b) \mapsto (a,b)\mathcal{H}_n.
$$
\end{definition}

\begin{remark}
We use the $\mathcal{H}_n$ as defined in \cite{Spr12} which differ slightly from the one in \cite{Spr16}.
\end{remark}

Recall that $F_\infty$ denotes the unramified $\Zp$-extension of $\Qp$. Put $U \colonequals \Gal (F_\infty / \Qp)$. Fix $(d_m)_m \in \varprojlim_m \mathcal{O}_{F_m}^\times$ a $\Lambda(U)$-basis of $\varprojlim_m \mathcal{O}_{F_m}$ as in \cite[Lemma 2.2]{Wan14}. Then, the arguments in \cite[Proposition 2.15]{Spr16} gives the following proposition:
\begin{proposition}\label{prop:2-variable chromatic}
There exists a unique homomorphism $\col_{n,m}:H^1(F_m^n,T) \to \frac{\Lambda_{n,m}^{\oplus 2}}{\ker h_{n,m}}$ such that $( P_{(n,m),c_{n,m}}, P_{(n,m),c_{n-1,m}}) = h_{n,m}\circ \col_{n,m}$.
\end{proposition}

Here, the system of points $c_{(n,m)}$ (denoted by $c_{(n,d_m)}$ in \cite{Spr16}) are trace-compatible points in the formal group $\widehat{E}(\mathfrak{m}_{n,m})$ constructed in \cite[Lemma 2.10]{Spr16}. We will not need their precise construction. To record one key property of $c_{(n,m)}$, we will need to introduce some notation. Let $\gamma$ (resp. $u$) be a generator of $\Gamma^\text{cyc}_{\Qp}$ (resp. $U$). We identify $\Gal(k_\infty/\Qp)$ with $\Gamma^\text{cyc}_{\Qp} \times U$.

\begin{lemma}
Let $\log_{\widehat{E}}$ be the formal logarithm on $\widehat{E}$ and let $\theta$ be a finite order character of the Galois group $\Gal(k_\infty/\Qp)$ such that $\theta(\gamma)$ is a primitive $p^n$th root of unity and $\theta(u)$ is a primitive $p^m$th root of unity. Then,
$$
\sum_{\sigma \in \Gal(F_m^n/\Qp)} \log_{\widehat{E}}(c_{(n,m)})^\sigma \theta(\sigma) = \tau(\theta\lvert_{G_n})\theta(u)^{n+1} \cdot \sum_{v \in \Gal(F_m/\Qp)} \theta (v)d_m^v
$$
where $\tau(\theta\lvert_{G_n})$ is the Gauss sum of $\theta$ restricted to $G_n$.
\end{lemma}

\begin{proof}
This can be found in the proof of \cite[Proposition 3.7]{Spr16}. Note that the exponent $n+1$ of $\theta(u)$ is missing in \textit{loc. cit.}
\end{proof}

Let $\exp^*: H^1(F_m^n,T\otimes \Qp) \to \mathrm{cotan}(E/F_m^n)$ be the dual exponential map. Let $\omega_E$ be the invariant differential of $E$. Then, $\mathrm{cotan}(E/F_m^n)$ is one dimensional with basis $\{ \omega_E \}$. Define $\exp_{\omega_E}^*(y)$ by the relation $\exp^*(y) = \exp_{\omega_E}^*(y) \omega_E$.

\begin{lemma}\label{lm:interpolation K-pairing}
Let $\theta$ be a finite order character of $\Gal(k_\infty/\Qp)$ such that $\theta(\gamma)$ is a primitive $p^n$th root of unity and $\theta(u)$ is a primitive $p^m$th root of unity. Then, $\begin{bmatrix} P_{(n,m),c_{n,m}}(z)(\theta) \\ P_{(n,m),c_{n-1,m}}(z)(\theta) \end{bmatrix}$ equals to
$$
\begin{bmatrix}\left( \tau(\theta\lvert_{G_n})\theta(u)^{n+1} \cdot \sum_{v \in \Gal(F_m/\Qp)} \theta (v)d_m^v\right) \left( \sum_{\sigma \in \Gal(F_m^n/\Qp)}\exp^\ast_{\omega_E}(z^\sigma)\theta(\sigma)^{-1}\right) \\ 0  \end{bmatrix}.
$$
\end{lemma}

\begin{proof}
This follows from the identity
$$
P_{(n,m),x}(z) = \left( \sum_{\sigma \in \Gal(F_m^n/\Qp)} \log_{\widehat{E}}(x^\sigma)\sigma \right)\left(  \sum_{\sigma \in \Gal(F_m^n/\Qp)} \exp^\ast_{\omega_E}(z^\sigma)\sigma^{-1}\right)
$$
and the fact that $c_{n-1,m}$ is defined with $p^{n}$th root of unity making the "Gauss sum" $\sum_{\sigma \in G_n}\zeta_{p^{n}}^\sigma \theta(\sigma)$ equals to $0$.
\end{proof}

Let $A \colonequals \begin{bmatrix} a_p & p \\ -1 & 0 \end{bmatrix}$ and put $\mathcal{H} \colonequals \lim_{n\to \infty}\mathcal{H}_nA^{-n-1}$. The entries of $\mathcal{H}$ converge on the open unit disc of $\mathbf{C}_p$ and are $O(\log_p (1+X)^{\frac{1}{2}})$ \cite[Lemma 4.4 and Lemma 4.8]{Spr12}. By \cite[Proposition 2.16 and Proposition 2.17]{Spr16}, the Coleman maps $\col_{n,m}$ are compatible in the unramified and totally ramified directions allowing one to assemble them into a map 
$$
(\col^\sharp,\col^\flat) \colonequals \varprojlim_n \varprojlim_m H^1(F_m^n,T) \to \varprojlim_n \varprojlim_m \frac{\Lambda_{n,m}^{\oplus 2}}{\ker h_{n,m}}.
$$
We denote by $\col_{chr}$ this vector of Coleman maps. By \cite[Lemma 2.18]{Spr16}, 
$$
\varprojlim_n \varprojlim_m \frac{\Lambda_{n,m}^{\oplus 2}}{\ker h_{n,m}} \cong \Lambda(\Gamma)^{\oplus 2}.
$$
Thus $\col_{chr}$ takes value in $\Lambda(\Gamma)^{\oplus 2}$. Moreover, $$\left( \col_{chr}(z) \mathcal{H}\right) (\theta) = (P_{(n,m),c_{n,m}}(z)(\theta), P_{(n,m),c_{n-1,m}}(z)(\theta)) A^{-n-1}$$ by proposition \ref{prop:2-variable chromatic}. In order to compare with the construction of multi-signed Coleman map, it will be easier to transpose every object in Sprung's construction. Thus we find
$$
\left( \mathcal{H}^t \col_{chr}^t(z)\right)(\theta) = (A^t)^{-n-1}(P_{(n,m),c_{n,m}}(z)(\theta), P_{(n,m),c_{n-1,m}}(z)(\theta))^t.
$$

We note that $A^t = \begin{bmatrix}a_p & p \\ -1 & 0 \end{bmatrix}$ and $\mathcal{A}_i^t = \begin{bmatrix} a_p & \Phi_i(1+X) \\ -1 & 0 \end{bmatrix}$.

\subsection{Comparison with the multi-signed theory}

The goal is to find a basis of $\Dcris(T)$ such that the matrix $M_T$ constructed in section \ref{definition} agrees with $\mathcal{H}$. Since $T=T^\ast(1)$, we identify $\Dcris(T)$ with $\Dcris(T^\ast(1))$. Let $\{\omega\}$ be a basis for $\Fil^0 \Dcris(T)$ such that $\omega$ corresponds to the Néron differential $\omega_E$ via the isomorphism
$$
\Fil^0\Dcris(V) \cong \Fil^0\mathbb{D}_{\mathrm{dR}}(V)\cong \mathrm{cotan}(E/\Qp).
$$
Let
$$
\exp^\ast_{F_m^n,T}: H^1(F_m^n,T^\ast(1)) \to \Fil^0\Dcris(T^\ast(1))
$$
be Bloch--Kato dual exponential map. Define $\exp^\ast_{F_m^n,\omega}$ by the relation $\exp^\ast_{F_m^n,T}(z) = \exp^\ast_{F_m^n,\omega}(z)\cdot \omega$. By \cite[Example 3.11]{BK90}, the Bloch--Kato exponential map coincides with the classical exponential map. Consider the basis $\{\omega, -\varphi(\omega) \}$ of $\Dcris(T)$. The operator $\varphi$ satisfies the polynomial $X^2-\frac{a_p}{p}X + \frac{1}{p}$, hence the matrix of $\varphi$ with respect to this basis is 
$$
C_\varphi = \begin{bmatrix} 0 & 1/p \\ -1 & a_p/p \end{bmatrix}.
$$
We denote by $\alpha$ and $\beta$ the roots of the above polynomial. Following the construction of section \ref{definition}, we find the matrices 
$$
C_n = \begin{bmatrix} a_p & -1 \\ \Phi_n(1+X) & 0 \end{bmatrix}.
$$
Remark that $C_n = \mathcal{A}_n^t$ and $C_\varphi = (A^t)^{-1}$.

\begin{lemma}\label{dual basis}
Let $\delta:= [\omega, \varphi(\omega)]$. The basis $\{ \varphi(\omega)/\delta, \omega/\delta \}$ is the dual basis of $\{\omega,-\varphi(\omega) \}$.
\end{lemma}

\begin{proof}
Because $\Fil^0 \Dcris(T)$ is its own orthogonal complement, we have $[\omega,\omega]=0$. We must have $[\omega, \varphi(\omega)] \in \Qp^\times$ since $[,]$ is non-degenerate. So it makes sense to consider $\varphi(\omega)/\delta$ and $\omega/\delta$. By using the fact that $[,]$ is alternating and $\Qp$-linear, we readily check that the two bases are dual to each other.
\end{proof}

Consider the basis $(d_m^{-1})_m \in \varprojlim_m \mathcal{O}^\times_{F_m}$ of $\varprojlim_m \mathcal{O}_{F_m}$ as a $\Lambda(U)$-module. Then, 
$$
y_{F_\infty/\Qp}((d_m^{-1})_m)=\left( \sum_{v \in \Gal(F_m/\Qp)} d_m^{-v} \cdot v \right)_m
$$
is a $\Lambda(U)$-basis of $S_{F_\infty/\Qp}$ which we denote by $\Omega_{\Qp}$. If $\{v_1,v_2\}$ is a Hodge-compatible basis of $\Dcris(T)$, we define the $v_i$th component of the $p$-adic regulator 
$$
\mathcal{L}_{T,k_\infty,v_i}:H^1_\mathrm{Iw}(k_\infty,T) \to \Lambda(U) \widehat{\otimes} \mathcal{H}(\Gamma^\text{cyc})
$$
as the composition of $\frac{\mathcal{L}_{T,k_\infty}}{\Omega_{\Qp}}$ with the projection of $\Lambda(U) \widehat{\otimes} \mathcal{H}(\Gamma^\text{cyc})\otimes_{\Zp} \Dcris(T)$ to the $v_i$-component. Note that if $\{v_1^\prime,v_2^\prime \}$ is the dual basis, then $\mathcal{L}_{T,k_\infty,v_i(z)} = \left[ \frac{\mathcal{L}_{T,k_\infty}(z)}{\Omega_{\Qp}}, v_i^\prime \right]$.

\begin{proposition}\label{prop:interpolation regulator}
Let $\theta$ be a finite order character of $\Gal(k_\infty/\Qp)$ such that $\theta(\gamma)$ is a primitive $p^n$th root of unity and $\theta(u)$ is a primitive $p^m$th root of unity. Then, the column vector $(\mathcal{L}_{T,k_\infty,\omega}(z)(\theta), \mathcal{L}_{T,k_\infty,\varphi(\omega)}(z)(\theta))^t$ is equal to
$$
p^{-n-1}\tau(\theta\lvert_{G_n})\theta(u)^{n+1}\theta(\Omega_{\Qp}^{-1}) \left( \sum_{\sigma \in \Gal(F_m^n/\Qp)} \theta^{-1}(\sigma)\exp^\ast_{F_m^n,\omega}(\sigma(z)) \right)
\begin{bmatrix}
\frac{-\beta\alpha^{n+1}+\alpha\beta^{n+1}}{-\alpha+\beta}\\
\frac{-\beta\alpha^{n+2}+\alpha\beta^{n+2}}{-\alpha+\beta}
\end{bmatrix}
$$
where $\theta(\Omega_{\Qp}^{-1}) = \sum_{v \in \Gal(F_m/\Qp)}\theta(v)d_m^v$.
\end{proposition}

\begin{proof}
The proof for $\mathcal{L}_{T,k_\infty,-\varphi(\omega)}(z)(\theta)$ is similar to the proof for $\mathcal{L}_{T,k_\infty,\omega}(z)(\theta)$ so we only show the result for the latter. Let $\varepsilon(\theta^{-1})$ be the epsilon-factor of $\theta^{-1}$. The discussion of \cite[Section 2.8]{LZ14} shows that (notice that in \textit{loc. cit.}, the Gauss sum $\tau(\theta\lvert_{G_n})$ is defined by $\sum_{\sigma \in \Gal(F_m^n/\Qp)}\theta^{-1}(\sigma) \zeta_{p^{n+1}}^\sigma$ which is different from the convention we are using)
$$
\varepsilon(\theta^{-1}) = \frac{p^n\theta(u)^{n+1}}{\tau(\theta^{-1}\lvert_{G_n})}.
$$
Moreover, we have that $\theta(\Omega_{\Qp}^{-1}) = \sum_{v \in \Gal(F_m/\Qp)}\theta(v)d_m^v$.
By \cite[Theorem 4.15]{LZ14}, 
$$
\mathcal{L}_{T,k_\infty}(z)(\theta) = \frac{\theta(u)^{n+1}}{\tau(\theta^{-1}\lvert_{G_n})}\varphi^{n+1}\exp^\ast_{T(\theta^{-1})^\ast (1)}(z).
$$
It follows that
$$
\mathcal{L}_{T,k_\infty,\omega}(z)(\theta) = \theta(\Omega_{\Qp}^{-1})\frac{\theta(u)^{n+1}}{\tau(\theta^{-1}\lvert_{G_n})} \left[ \sum_{\sigma \in \Gal(F_m^n/\Qp)}\theta^{-1}(\sigma)\exp_{F_m^n,T}^\ast(\sigma(z)),\varphi^{-n-1}(\varphi(\omega)/\delta) \right]
$$
by using the identity 
$$
\exp^\ast_{\Qp,T(\theta^{-1})^\ast(1)}(z) = \sum_{\sigma \in \Gal(F_m^n/\Qp)}\theta^{-1}(\sigma)\exp^\ast_{F_m^n,T^\ast(1)}(\sigma(z)). 
$$
Powers of the matrix of $\varphi$ with respect to the basis $\{\omega, -\varphi(\omega)\}$ are given by
$$
C_\varphi^{-k} = \frac{1}{-\alpha+\beta} \begin{bmatrix} -\alpha^{k+1}+\beta^{k+1} & \alpha^k -\beta^k \\ -\beta\alpha^{k+1} +\alpha \beta^{k+1} & \beta \alpha^k - \alpha \beta^k \end{bmatrix}.
$$

Recall that the kernel of the exponential $\exp_{F_m^n,T} : F_m^n\otimes \Dcris(T) \to H^1(F_m^n,T)$ is $\Fil^0\Dcris(T)$. So $\exp_{F_m^n,T}(\omega)=0$. Let $y$ be any element of $H^1(F_m^n,T)$. Then,
$$
[\exp_{F_m^n,T}^\ast (y), \omega ] = \langle y,\exp_{F_m^n,T}(\omega)\rangle =0. 
$$
Therefore, only the terms in $-\varphi(\omega)$ in the formula for the regulator will contribute non trivially. Therefore,
\begin{align*}
&\frac{\tau(\theta^{-1}\lvert_{G_n})\mathcal{L}_{T,k_\infty,\omega}(z)(\theta)}{\theta(\Omega_{\Qp}^{-1})\theta(u)^{n+1}}\\
&=\left[ \sum_{\sigma \in \Gal(F_m^n/\Qp)} \theta^{-1}(\sigma)\exp^\ast_{F_m^n,T}(\sigma(z)), \frac{\delta^{-1}(-\beta\alpha^{n+1}+\alpha\beta^{n+1})}{-\alpha+\beta}\varphi(\omega) \right]\\
&= \frac{(-\beta\alpha^{n+1}+\alpha\beta^{n+1})}{-\alpha+\beta} \left( \sum_{\sigma \in \Gal(F_m^n/\Qp)} \theta^{-1}(\sigma)\exp^\ast_{F_m^n,\omega}(\sigma(z)) \right)
\end{align*}
by duality. The result follows since
$$
\frac{1}{\tau(\theta^{-1}\lvert_{G_n})} = p^{-n-1}\overline{\theta}(-1)\tau(\theta\lvert_{G_n}) 
$$
and $\overline{\theta}(-1)=1$ because the cyclotomic part of $\theta$ factors through $$\Z/(p-1)\Z \times \Zp \to \Zp \to \Z/p^n\Z.$$
\end{proof}

\begin{lemma}\label{lm:two-variable interpolation}
The maps $\mathcal{H}^t \col_{\text{chr}}^t$ and $\mathcal{L}_{T,k_\infty}$ are uniquely determined by their values at finite order characters of $\Gamma_p$.
\end{lemma}

\begin{proof}
Any finite order character $\theta$ of $\Gamma_p$ can be written as $\theta = \eta \varpi$ where $\eta$ is a finite order character of $\Gamma^\mathrm{cyc}_{\Qp}$ and $\varpi$ a finite order character of $U$. For a fixed $\varpi$, $\mathcal{H}^t \col_{\text{chr}}^t(\varpi)$ is $O(\log(1+X)^{1/2})$ by \cite[Lemma 2.18 and Lemma 2.34]{Spr16} and thus is uniquely determined by its values at $\eta$. For $\mathcal{L}_{T,k_\infty}$, property (1) of \cite[Theorem 4.7]{LZ14} uniquely determines the values since Perrin-Riou's regulator $\mathcal{L}_{T,F_m}$ is also $O(\log(1+X)^{1/2})$. For $\eta$ fixed, $\mathcal{H}^t \col_{\text{chr}}^t(\eta)$ (resp. $\mathcal{L}_{T,k_\infty}(\eta)$) is an element of $\Zp\llbracket U \rrbracket$ (resp. $\Omega_{\Qp}\cdot \Zp\llbracket U \rrbracket$) and therefore uniquely characterized by its values at such finite order character $\varpi$.
\end{proof}

\begin{proposition}\label{prop:comparison}
We have $\col^t_{\text{chr}} = -\col_T^{k_\infty}$.
\end{proposition}

\begin{proof}
Let $\theta$ be a finite order character of $\Gal(k_\infty/\Qp)$ such that $\theta(\gamma)$ is a primitive $p^n$th root of unity and $\theta(u)$ is a primitive $p^m$th root of unity. Combining lemma \ref{lm:interpolation K-pairing} and proposition \ref{prop:interpolation regulator}, we find
\begin{align*}
    \mathcal{H}^t\cdot \col_{\text{chr}}^t(z)(\theta) &= (A^t)^{-n-1} \begin{bmatrix}P_{(n,m),c_{n,m}}(z)(\theta) \\ P_{(n,m),c_{n-1,m}}(z)(\theta) \end{bmatrix} \\
    &= C_\varphi^{n+1}\begin{bmatrix}P_{(n,m),c_{n,m}}(z)(\theta) \\ P_{(n,m),c_{n-1,m}}(z)(\theta) \end{bmatrix}\\
    &= -\begin{bmatrix}\mathcal{L}_{T,k_\infty,\omega}(z)(\theta) \\ \mathcal{L}_{T,k_\infty,\varphi(\omega)}(z)(\theta) \end{bmatrix}
\end{align*}
where we used the relation $\alpha\beta=p$ to rewrite $C_\varphi^{n+1}$ as
$$
\frac{p^{-n-1}}{-\alpha+\beta}\begin{bmatrix} \beta\alpha^{n+1}-\alpha\beta^{n+1} & \ast\\ \beta\alpha^{n+2}-\alpha\beta^{n+2} & \ast\end{bmatrix}.
$$
Lemma \ref{lm:two-variable interpolation} and the formula for the decomposition of the regulator (see proposition \ref{coleman inverse limit}) give $\mathcal{H}^t\cdot \col^t_{\text{chr}}(z) = -M_T \cdot \col_{T}^{k_\infty}$. The result follows since by our choice of basis, $M_T=\mathcal{H}^t$ and $M_T$ is invertible by proposition \ref{determinant}.
\end{proof}

\begin{proposition}\label{chromatic orthogonal}
The submodule $\ker \col_{\sharp/\flat}$ is its own orthogonal complement with respect to the Perrin-Riou pairing
$$
\langle \sim,\sim \rangle:H^1_\mathrm{Iw}(k_\infty,T) \times H^1_\mathrm{Iw}(k_\infty,T^\ast(1)) \to \Lambda(\Gamma_p).
$$
\end{proposition}

\begin{proof}
By lemma \ref{orthogonal}, $\ker \col_{T, \omega}$ and $\ker \col_{T,\varphi(\omega)}$ are their own orthogonal complement. The result follows from proposition \ref{prop:comparison}.
\end{proof}

\subsection{The functional equation}

Let $\Sel(E/K_\infty)$ be the classical $p^\infty$-Selmer group of $E$ over $K_\infty$. Also denote by $E_{\sharp/\flat}$ the exact annihilator of $\ker \col_{\sharp/\flat}$ under Tate local duality
$$
H^1_\mathrm{Iw}(k_\infty,T) \times H^1(k_\infty,E[p^\infty]) \to \Qp/\Zp.
$$

\begin{definition}
For $\star,\circ \in \{\sharp,\flat\}$, the $\star\circ$-Selmer group of $E$ over $K_\infty$ is defined by
$$
\Sel^{\star\circ}(E/K_\infty) \colonequals \ker \left( \Sel(E/K_\infty) \to \frac{H^1(K_{\infty,\gp}, E[p^\infty])}{E_\star} \oplus \frac{H^1(K_{\infty,\gp^c},E[p^\infty])}{E_\circ} \right).
$$
\end{definition}

\begin{remark}
Our definition of $\Sel^{\star\circ}(E/K_\infty)$ differs from the one in \cite{Spr16} where different matrices are used to define $\mathcal{H}$.
\end{remark}

\begin{theorem}\label{main chromatic}
The functional equation
$$
\Sel^{\star\circ}(E/K_\infty)^\vee \sim \Sel^{\star\circ}(E/K_\infty)^{\vee,\iota}
$$
holds.
\end{theorem}

\begin{proof}
Define $\Sel^{\star\circ}(E/K_n)$ as the kernel of the map
$$
\Sel(E/K_n) \to \frac{H^1(K_{n,\gp}, E[p^\infty])}{E_\star^{\Gamma_{p,n}}} \oplus \frac{H^1(K_{n,\gp^c},E[p^\infty])}{E_\circ^{\Gamma_{p,n}}}
$$
where $\Sel(E/K_n)$ is the classical $p^\infty$-Selmer group over $K_n$. 
We apply the same arguments as in section \ref{sct:Non-simple} with 
$$
\mathfrak{a}_n = \mathfrak{b}_n = \Sel^{\star\circ}(E/K_n)_{/\text{div}}
$$
to show equality of non-simple parts. Remark that proposition \ref{chromatic orthogonal} and lemma \ref{ortho col} applied instead to $E_{\sharp/\flat}^{\Gamma_{p,n}}$ and $\overline{(\ker \col_{\sharp/\flat})_n}$ gives us the orthogonality condition required for the Flach pairing to be perfect. We directly get the equality of simple part by lemma \ref{involution simple}.
\end{proof}

\begin{remark}
Let $\mathbf{Q}^\text{cyc}/\mathbf{Q}$ be the cyclotomic $\Zp$-extension and let $\Sel^{\sharp / \flat}(E/\Q^\text{cyc})$ be the one-variable Selmer group defined in \cite{Spr12}. We can apply proposition \ref{chromatic orthogonal} and theorem \ref{main chromatic} verbatim to the one-variable setting. Thus one also get the functional equation
$$
\Sel^{\sharp / \flat}(E/\Q^\text{cyc})^\vee \sim \Sel^{\sharp / \flat}(E/\Q^\text{cyc})^{\vee,\iota}.
$$
\end{remark}

In \cite[Theorem 5.19]{Spr17}, Sprung proved an analytic functional equation for the completed chromatic $p$-adic $L$-function $\widehat{L}_p^{\sharp/\flat}(f,X)$ attached to a modular form of weight two. In the terminology of \textit{loc. cit.}, the chromatic Selmer groups considered in the present paper are the \textit{non-completed} chromatic Selmer groups. The previous remark hints at a functional equation for the non-completed $p$-adic $L$-function $L_p^{\sharp/\flat}(f,X)$. It would be interesting to find the explicit factor appearing in the functional equation.

\bibliographystyle{acm}
\bibliography{main.bib}

\end{document}